\documentclass[a4paper,12pt,reqno]{amsart}
\usepackage{amssymb, amsmath, amsthm,amssymb,amsfonts,latexsym}
\usepackage{tikz-cd}
\usepackage{extarrows}
\usepackage{mathrsfs}
\usepackage{arydshln}
\usepackage[figuresleft]{rotating}
\usepackage{multirow}

\usepackage{enumerate}
\usepackage{booktabs,float,bm,bbm}
\usepackage{xfrac}

\usepackage[square, comma, sort&compress, numbers]{natbib}
\usepackage[backref=page,
           colorlinks = true, linkcolor = blue, citecolor = teal,
           urlcolor  = black,  anchorcolor = black]{hyperref}

\numberwithin{equation}{section}
\theoremstyle{plain}
\newtheorem{theorem}{Theorem}[section]
\newtheorem{proposition}[theorem]{Proposition}
\newtheorem{corollary}[theorem]{Corollary}
\newtheorem{lemma}[theorem]{Lemma}

\theoremstyle{definition}

\newtheorem{example}[theorem]{Example}

\theoremstyle{remark}
\newtheorem{remark}[theorem]{Remark}

\newcommand{\Z}{\ensuremath{\mathbb{Z}}}

\newcommand{\C}{\mathbb{C}}
\newcommand{\Q}{\mathbb{Q}}
\newcommand{\HP}[1]{\mathbb{H}P^{#1}}
\newcommand{\CP}[1]{\mathbb{C}P^{#1}}
\newcommand{\caL}{\mathcal{L}}

\newcommand{\z}[1]{\ensuremath{\Z/2^{#1}}}
\newcommand{\zp}[1]{\ensuremath{\Z/p^{#1}}}
\newcommand{\ZZ}[1]{\ensuremath{\mathbb{Z}_{(#1)}}}

\newcommand{\PP}{\ensuremath{\mathcal{P}^1}}

\newcommand{\xra}{\ensuremath{\xrightarrow}}

\newcommand{\id}{\ensuremath{\mathbbm{1}}}

\newcommand{\ol}{\ensuremath{\overline}}

\newcommand{\wtd}{\ensuremath{\widetilde}}

\newcommand{\lra}[1]{\ensuremath{\langle #1\rangle}}

\DeclareMathOperator{\im}{im}
\DeclareMathOperator{\cok}{coker}
\DeclareMathOperator{\Aut}{Aut}

\newcommand{\mat}[4]{\ensuremath{\footnotesize{\begin{bmatrix}
				#1&#2\\
				#3&#4
\end{bmatrix}}}}

\newcommand{\matwo}[2]{\ensuremath{\footnotesize{\begin{bmatrix}
				#1\\
				#2
\end{bmatrix}}}}

 \newcommand{\smatwo}[2]{\ensuremath{\left[\begin{smallmatrix}
  #1\\
  #2
 \end{smallmatrix}\right]}}

\title[Cohomotopy Sets of $(2n+2)$-manifolds]{Cohomotopy Sets of $(n-1)$-connected $(2n+2)$-manifolds for small $n$}

\author[P. Li]{Pengcheng Li}
\address{Department of Mathematics, School of Sciences, Great Bay University,  {\rm 523000}, Dongguan, Guangdong Province, China}
\email{lipcaty@outlook.com}

\author[J. Pan]{Jianzhong Pan}
\address{
Institute of Mathematics,  Chinese Academy of Mathematics and Systems Science,  University of Chinese Academy of Sciences, {\rm 100190}, Beijing, China}
\email{pjz@amss.ac.cn}

\author[J. Wu]{Jie Wu}
\address{Hebei Normal University, Shijiazhuang, Hebei Province, \rm{050024}, China; Beijing Key Laboratory of Topological Statistics and Applications for Complex Systems, Beijing Institute of Mathematical Sciences and Applications  \rm{101408}, Beijing, China}
\email{wujie@bimsa.cn}


\begin{document}

\begin{abstract}
   Let $M$ be a closed orientable $(n-1)$-connected $(2n+2)$-manifold, $n\geq 2$. In this paper we combine the Postnikov tower of spheres and the homotopy decomposition of the reduced suspension space $\Sigma M$ to investigate the (integral) cohomotopy sets $\pi^\ast(M)$ for $n=2,3,4$, under the assumption that $M$ has $2$-torsion-free homology. All cohomotopy sets $\pi^i(M)$ of such manifolds $M$ are well-understood except $\pi^4(M)$ for $n=3,4$.
 \end{abstract}
 
 \maketitle

 \section{Introduction}
Given a based CW complex $X$ and a positive integer $n$, the $n$-th cohomotopy set $\pi^n(X)$ consists of the homotopy classes of based maps from $X$ into the $n$-sphere, which admits a natural abelian group structure in the stable range $\dim(X)\leq 2n-2$. The study of stable cohomotopy groups $\pi^n(X)$ can be traced back to Borsuk \cite{Borsuk36} and Peterson \cite{Peterson56-2}. 
Due to the Pontrjagin-Thom construction (cf.~\citep[Chapter IX, Theorem 5.5]{Kosinski93}), the cohomotopy sets of closed smooth manifolds has been extensively studied, such as \cite{Hansen81,Taylor2012,KMT2012,Konstantis2020,Kons2020}. In particular, Taylor \cite{Taylor2012} solved the Steenrod problem (\cite{Steenrod1947}) by proving that for a based CW-complex $X$ of dimension at most $n+1$, $n\geq 3$, the short exact sequence of abelian groups  
 \[0\to H^{n+1}(X,\z{})/Sq^2_\Z(H^{n-1}(X;\Z))\xra{}\pi^n(X)\xra{h^n}H^n(X)\to 0,\]
 is split if and only if the integral Steenrod square 
 \[Sq^2_\Z\colon H^{n-1}(X;\Z)\xra{\rho_2}H^{n-1}(X;\z{})\xra{Sq^2}H^{n+1}(X;\z{})\]  
has the same image as the classical Steenrod square $Sq^2$.
Here the homomorphism $h^n$ is the $n$-th cohomotopy Hurewicz map sending a class $f\in\pi^n(X)$ to $f^\ast(\sigma_n)$ with $\sigma_n\in H^n(S^n)$ a generator. 
Furthermore, for a (closed oriented) connected $4$-manifold $M$, Taylor enumerated the second cohomotopy set $\pi^2(M)$ by a homotopy theoretic approach, while Kirby, Melvin and Teichner \cite{KMT2012} characterized all the cohomotopy sets of $M$ by geometric arguments. Recently, Li and Zhu \cite{LZ-5mfld}, Amelotte, Culter and So \cite{ACS24} independently characterized the cohomotopy sets of connected $5$-manifolds (with  $H_1(M)$  $2$-torsion-free) in terms of the homotopy decomposition of the reduced suspension space of manifolds. On the other hand, the authors comprehensively studied Peterson's generalized cohomotopy groups $\pi^n(X;G)$ with coefficients in a cyclic group $G$ in terms of unstable homotopy theory. Research has also been done on the motivic stable cohomotopy groups of smooth schemes, which are closely related to the Euler class groups and the orbit sets of unimodular rows of certain group action,   see \cite{AF22, Lerbet24}.

The purpose of this paper is to explore the algebraic topology of $(n-1)$-connected $(2n+2)$-manifolds from the aspect of cohomotopy sets.  Inspired by  C.T.C. Wall's classification of $(n-1)$-connected $2n$- and $(2n+1)$-manifolds \cite{Wall62, Wall67}, the classification of the manifolds in question were extensively and deeply studied by geometric topologists.
For instance, Ishimoto \cite{Ishimoto73} classified the homotopy type of $(n-1)$-connected $(2n+2)$-manifolds with torsion-free homology up to diffeomorphism mod a homotopy $(2n+2)$-sphere, 
Sasao and Takahashi \cite{ST79} investigated the homotopy type of $(2n+2)$-dimensional rational homology spheres, and Fang and Pan \cite{FP04} discovered complete invariants for classifying $(n-1)$-connected $(2n+2)$-dimensional $\pi$-manifolds with general homology.

Let $M$ be a closed orientable $(n-1)$-connected $(2n+2)$-manifold. By Poincar\'e duality, the integral homology groups $H_\ast(M)$ are given by the following table:
\begin{equation}\label{table:HM}
	\begin{tabular}{cccccccc}
		\toprule
		$i$& $n$&$n+1$&$n+2$&$0,2n+2$&$\text{otherwise}$
		\\\midrule
		$H_i(M)$& $\Z^l\oplus T$&  $\Z^k\oplus T$ & $\Z^l$& $\Z$&$0$
		\\ \bottomrule
	\end{tabular}  
\end{equation}
where $l,k\geq 0$ are integers and $T$ is a finite torsion abelian group.
Combining Postnikov stages of spheres and the homotopy type of the reduced suspension $\Sigma M$ described by Theorem \ref{thm:SM} below, we characterize the cohomotopy sets of $(n-1)$-connected $(2n+2)$-manifolds for $n=2,3,4$ by the following three theorems, respectively.

\begin{theorem}\label{thm:chtp-6mfd}
   Let $M$ be a simply connected $6$-manifold with $H_\ast(M)$ given by (\ref{table:HM}), $n=2$. Set $\varepsilon=0$ if $M$ is spin, otherwise $\varepsilon=1$. 
   \begin{enumerate}[(1)]
      \itemsep=1pt
      \item\label{chtp=5} $\pi^5(M)\cong \z{1-\varepsilon}$, $\pi^6(M)\cong \Z$ and $\pi^i(M)=0$ for $i>6$ or $i=1$.
      \item\label{chtp=4} There is a short exact sequence of abelian groups 
	 \[0\to \z{1-\varepsilon}\to \pi^4(M)\xra{\imath}\ker(Sq^2_\Z)\to 0,\] 
	 which splits  if $T$ is $2$-torsion-free.
	 Here $Sq^2_\Z$ is the integral Steenrod square operation,  and $\imath$ is the homomorphism such that the composition 
	 \[h^4\colon \pi^4(M)\xra{\imath}\ker(Sq^2_\Z)\hookrightarrow H^4(M)\]
      is the fourth cohomotopy Hurewicz homomorphism.

    \item\label{chtp=3}  Let $c$ be the integer given by Theorem \ref{thm:SM} with $n=2$.  If $T$ is $2$-torsion-free, then there is a short exact sequence of groups 
	\[0\to G_{12}\oplus T\to \pi^3(M)\to \Z^k\oplus \bigoplus_{i=1}^{l-c-\varepsilon}\z{}\to 0,\]
	where $G_{12}$ is a subgroup of $\Z/12$. Moreover, the homomorphism $\pi^3(M)\to \bigoplus_{i=1}^{l-c-\varepsilon}\z{}$ admits a splitting section, and the action of $ \Z^k\oplus \bigoplus_{i=1}^{l-c-\varepsilon}\z{}$ on $T$ is trivial. 
    
   \item\label{chtp=2} The natural action of $\pi^3(M)$ on $\pi^2(M)$ is transitive and free, hence there is a bijection between $\pi^3(M)$ and $\pi^2(M)$ as sets. Moreover, a class $u\in H^2(M)$ lifts to a class $\tilde{u}\in \pi^2(M)$ by the second cohomotopy Hurewicz map if and only if 
   \[u^2=0,\quad (1-\varepsilon)\Theta_0(u)=0,\]
	where $\Theta_0\colon \{u\in H^2(M;\z{})\mid u^2=0\} \to H^6(M;\z{}) $ is the restriction of some secondary cohomology operation $\Theta$, see Proposition \ref{prop:chtp=2}.
	 
   \end{enumerate}
\end{theorem}

\begin{theorem}\label{thm:chtp-8mfd}
	Let $M$ be a closed orientable smooth $2$-connected $8$-manifold with $H_\ast(M)$ given by (\ref{table:HM}), $n=3$. 
	\begin{enumerate}[(1)]
		\itemsep=1pt
		\item\label{cohtp=7,6} $\pi^6(M)\cong\pi^7(M)\cong\z{}$, $\pi^8(M)\cong \Z$ and $\pi^i(M)=0$ for $i>8$ or $i=1$.
		\item\label{cohtp=2} The induced map $\eta_\sharp\colon\pi^3(M)\to \pi^2(M)$ is a bijection.
		\item\label{cohtp=5}  If $T$ is $2$-torsion-free,  then there is a split short exact sequence of abelian groups
		\[0\to G_{24}\to \pi^5(M)\xra{\jmath}\ker((\ol{Sq^2})_\sharp)\to 0, \]
		where $G_{24}$ is a subgroup of $\Z/24$, $\ker((\ol{Sq^2})_\sharp)\cong H^5(M)$ and $\jmath$ is a homomorphism such that the fifth cohomotopy Hurewicz homomorphism $h^5\colon \pi^5(M)\to H^5(M)$ factors through $\jmath$.
		\item\label{cohtp=3} Let $c$ be the integer given by Theorem \ref{thm:SM} with $n=3$ and let $T$ be $2$-torsion-free. Then there is an isomorphism 
		\[\pi^3(M)\cong H^3(M)\oplus (\z{})^{l-c}\oplus G,\]
	where $G$ is an abelian group embedding in the short exact sequence
	\[0\to \z{}\to G\to \bigoplus_{i=2}^k\z{}\oplus \z{\delta}\to 0,\]
	where $\delta\in\{0,1\}$.
	
	\end{enumerate}

\end{theorem}

\begin{theorem}\label{thm:chtp-10mfd}
	Let $M$ be a $3$-connected 10-manifold with $H_\ast(M)$ given by (\ref{table:HM}), $n=4$.
	\begin{enumerate}[(1)]
		\itemsep=1pt
		\item $\pi^8(M)\cong\pi^9(M)\cong\z{}$, $\pi_{10}(M)\cong\Z$ and $\pi_i(M)=0$ for $i>10$ or $i=1$.
		\item The induced map $\eta_\sharp\colon \pi^3(M)\to \pi^2(M)$ is a bijection.
		\item If $T$ is $2$-torsion-free, then $\pi^7(M)\cong G_{24}$ is a subgroup of $\Z/24$, and $\pi^6(M)$ is isomorphic to the kernel of the integral reduced Steenrod  $3$-th power $\PP_\Z\colon H^6(M;\Z)\to H^{10}(M;\Z/3)$. 
		\item\label{MS5} Let $u\in\pi^5(M)$ be a preimage of an element $\bar{u}\in \pi^6(\Sigma M)$ under the suspension map $E_\sharp\colon \pi^5(M)\to \pi^6(\Sigma M)$. If $T$ is $2$-torsion-free, then there is a bijection between $E_\sharp^{-1}(u)$ and $\pi^9(M)\cong\z{}$. 
	
		\item Let $c$ be given by Theorem \ref{thm:SM} with $n=4$ and let $T$ be $2$-torsion-free. Then there is an isomorphism  
		\[\pi^3(M)\cong \big(\bigoplus_{i=1}^{k+l-c}\z{}\big)\oplus\big(\bigoplus_{i=1}^{l-c}\Z/12\big)\oplus \big(\bigoplus_{i=1}^c\Z/6\big)\oplus (T\otimes \Z/3).\]   
	\end{enumerate}
\end{theorem}

The canonical $H$-multiplication of $S^3$ induces a group structure on $\pi^3(M)$, and it is surprising that the group structure of $\pi^3(M)$ of the $(2n+2)$-manifolds in question are abelian for $n=3,4$, under the assumption that $H_\ast(M)$ is $2$-torsion-free.  In Theorem \ref{thm:chtp-10mfd} (\ref{MS5}), we enumerate the fifth cohomotopy set $\pi^5(M)$ of a $3$-connected $10$-manifold $M$ by applying Taylor's analysis method \cite{Taylor2012} to the classical EHP fibre sequence on $S^5$. More generally, the EHP fibre sequence on odd-dimensional $i$-spheres is useful to enumerate the unstable cohomotopy sets $\pi^i(M)$ for odd $i$, say $\pi^5(M)$ of a $9$-manifold $M$; the details are omitted in this paper. On the other hand, we haven't found an effective way to study the fourth cohomotopy set $\pi^4(M)$ of an $8$- or $10$-manifold $M$.

\begin{remark}
Let $M$ be a closed orientable $2$-connected $8$-manifold. It is difficult to apply the quaternionic Hopf fibration sequence 
\[S^3\to S^7\xra{\nu_4}S^4\xra{\imath}B S^3\]
to enumerate the integral fourth cohomotopy set $\pi^4(M)$. However, when localized at an odd prime $p\geq 5$, we have the following characterization of the $p$-local cohomotopy set $\pi^4(M;\ZZ{p})=[M,S^4_{(p)}]$, see \cite[Theorem 5.12]{HL-7mfd}:
\begin{enumerate}
\item The Hurewicz map $h_{(p)}=\imath_\sharp\colon \pi^4(M;\ZZ{p})\to H^4(M;\ZZ{p})$ is onto the subset 
\[\mathcal{S}=\{\beta\in H^4(M;\ZZ{p})\mid \beta^2=\beta\smallsmile \beta=0\}.\]
\item For $\beta\in\mathcal{S}$, there is a bijection between $h_{(p)}^{-1}(\beta)$ and the cokernel of the homomorphism 
\[\psi_e\colon H^3(M;\ZZ{p})\to H^7(M;\ZZ{p}),\quad \psi_e(\gamma)=\gamma\smallsmile \delta_e,\]
where  $e\in \pi^4(M;\ZZ{p})$ satisfies $h_{(p)}(e)=\beta$, and $\delta_e=e^\ast(\iota)\in H^4(M;\ZZ{p})$ with $\iota\in H^4(S^4_{(p)};\ZZ{p})$. 
\end{enumerate}

\end{remark}

The computation of the cohomotopy groups $\pi^i(M)$ depends on the homotopy decomposition of the reduced suspension space $\Sigma M$, see Theorem \ref{thm:SM} below. 
We need the following notation. Denote by $C_{f}$ the homotopy cofibre of a map $f\colon X\to Y$. For integers $m,n\geq 1$,  let $P^n(G)$ be the Peterson space (or Moore space) which has a unique reduced integral cohomology group $G$ in dimension $n$; in particular, we write $P^n(m)=P^n(\Z/m)=S^{n-1}\cup_{m}e^n$ and set $P^n(1)=\ast$. Denote by $C^{n+2}_\eta=\Sigma^{n-2}\CP{2}$ the iterated suspension space of the complex projective plane.

\begin{theorem}\label{thm:SM}
	Let $M$ be a closed orientable smooth $(n-1)$-connected $(2n+2)$-manifold with $H_\ast(M)$ given by (\ref{table:HM}), where $T$ is a finite $2$-torsion-free group.   There exist non-negative integers $c\leq l$, $r_{j_0}$ and $r_{j_1}$ that depend on $M$; $c=0$ if and only if $Sq^2$ acts trivially on $H^n(M;\z{})$, and  $r_{j_0}=r_{j_1}=0$ if $T$ contains no $3$-torsion. 
\begin{enumerate}[(1)]
	\item\label{SM:6mfd}[Proposition \ref{prop:SM-hbar}] For $n=2$, there is a homotopy equivalence 
	\begin{equation*}
	   \Sigma M\simeq \bigvee_{i=2}^{l-c}(S^3\vee S^5)\vee \bigvee_{i=1}^kS^4\vee P^4\big(\frac{T}{\Z/3^{r_{j_0}}}\big)\vee P^5(T)\vee \bigvee_{i=2}^{c}C^5_\eta\vee C_\hbar, 
	\end{equation*}
	where $\hbar\colon S^6\to S^3\vee S^5\vee P^4(3^{r_{j_0}})\vee C^5_\eta$ refers to (\ref{eq:hbar}).
	\item\label{SM:8mfd}[Proposition \ref{prop:SM-8mfd-phi}] For $n=3$,  there is a homotopy equivalence 
	\begin{multline*}
		\Sigma M\simeq \bigvee_{i=2}^k S^5\vee \bigvee_{i=2}^{l-c}S^{4} \vee \bigvee_{i=1}^{l-c} S^6\vee \bigvee_{i=2}^c C^6_\eta \vee P^{5}(\frac{T}{\Z/3^{r_{j_0}}}) \vee P^{6}(\frac{T}{\Z/3^{r_{j_1}}})\vee  C_{\phi},
	\end{multline*}
	where $\phi\colon S^8\to S^5\vee S^4\vee C^6_\eta\vee P^5(3^{r_{j_0}})\vee P^6(3^{r_{j_1}})$ refers to (\ref{eq:phi}).
	\item\label{SM:10mfd}[Proposition \ref{prop:SM-10mfd}] For $n=4$, there is a homotopy equivalence 
	\begin{multline*}
		\Sigma M\simeq \bigvee_{i=1}^k S^6\vee \bigvee_{i=2}^{l-c}(S^{5}\vee S^7) \vee \bigvee_{i=3}^c C^7_{\eta} \vee P^{6}(T)\vee P^{7}(\frac{T}{\Z/3^{r_{j_0}}}) \vee C_{\varphi},
	\end{multline*}
	where $\varphi\colon S^{10}\to S^7\vee S^5\vee C^7_\eta\vee C^7_\eta\vee P^7(3^{r_{j_0}})$ refers to (\ref{eq:varphi}).
\end{enumerate}	

\end{theorem}

The integer $c$ in Theorem \ref{thm:SM} is dependent on the manifold $M$, and adjustments to the codomains of the attaching maps $\hbar$, $\phi$ and $\varphi$ are necessary when $c = l$ or $c = 0$. In particular, we adopt the standard convention that $\bigvee_{i=a}^b X_i=\ast$ if $b < a$. When the group $T$ contains 3-torsion, the integers $r_{j_0}$ and $r_{j_1}$ in Theorem \ref{thm:SM} can be determined through certain higher order cohomology operations, the precise description of which is similar that in theorems of \cite{LZ-5mfld, lipc-4mfld, HL-7mfd} and we omit it here.

The double suspension $\Sigma^2M$ of a simply connected 6-manifold $M$ \emph{without 2- or 3-torsion in homology} was firstly given by \cite[Theorem 1.1]{Huang-6mflds}, so Theorem \ref{thm:SM} (\ref{SM:6mfd}) is an improvement of Huang's theorem. 
Huang also determined the suspension splitting of $(n-1)$-connected $(2n+2)$-manifolds \emph{with torsion-free homology} in the cases $n=5$ and $n=7$, see \cite[Lemmas 6.4 and 6.6]{Huang22}, respectively. The suspension splitting of manifolds can be employed to compute important invariants in geometry and mathematical physics, such as the reduced complex or real $K$-groups and the gauge groups, see \cite{Huang21,Huang22, Huang-6mflds,LZ-5mfld, ST19, ACS24}.


The paper is arranged as follows. Section \ref{sec:prelim} contains globally used notation and some lemmas that describe some homotopy groups and their generators of elementary $(n-1)$-connected $(n+2)$-dimensional complexes with $2$-torsion-free homology; most material is known. In Section \ref{sec:chtp-stable} we study the stable cohomotopy groups in terms of the Postnikov tower of $S^n$ in the stable range; in particular, we give the characterization of stable cohomotopy groups of $(2n+2)$-manifolds, see Corollary \ref{cor:chtp:2n+2}.  Section \ref{sec:chtp=2n+2} is divided into three subsections to investigate the unstable cohomotopy sets of $(n-1)$-connected $(2n+2)$-manifolds for $n=2,3,4$, respectively. The proofs of Theorem \ref{thm:chtp-6mfd}, \ref{thm:chtp-8mfd} and \ref{thm:chtp-10mfd} are arranged at the end of Subsection \ref{sec:chtp:n=2}, \ref{sec:chtp:n=3} and \ref{sec:chtp:n=4}, respectively. The proofs utilize the homotopy decomposition of $\Sigma M$ given by Theorem \ref{thm:SM}, whose proof is arranged in Section \ref{sec:SM-proofs}.

\medskip
\noindent\textbf{Acknowledgements.}
The authors would like to thank the anonymous referee for their careful reading and valuable comments that improved the quality of this paper.
Pengcheng Li and Jianzhong Pan were partially supported by the National Natural Science Foundation of China (Grant No. 12101290 and 11971461), Jie Wu was partially supported by the High-level Scientific Research Foundation of Hebei Province.

\section{Global notation and lemmas}\label{sec:prelim}
Throughout this paper, all spaces are based CW-complexes and all maps are base-point-preserving. For spaces $X,Y$, let $[X,Y]$ be the set of homotopy classes of based maps from $X$ to $Y$; we don't distinguish a map with its homotopy class in notation. 
If there is no confusion, we write $\eta=\eta_n=\Sigma^{n-2}\eta$ as the iterated suspension of the Hopf map $\eta\in \pi_3(S^2)$. 
Denote by $\id_X$ the identity map of $X$; in particular, write $\id_n=\id_{S^n}$ and $\id_{\eta}=\id_{C^{n+2}_\eta}$ for simplicity. 
We denote by $i_{n+l}\colon S^{n+l}\to X$ and $q_{n+k}\colon Y\to S^{n+k}$ the possible canonical inclusion and pinch maps, respectively; we also add some superscripts to distinguish these maps. For example, we have the homotopy cofibre sequences for $P^{n}(p^r)$ and $C^{n+2}_\eta=\Sigma^{n-2}\C P^2$, which define the canonical inclusion and pinch maps: 
\begin{align*}
	S^{n-1}\xra{p^r}S^{n-1}\xra{i_{n-1}}P^{n}(p^r)\xra{q_n}S^n,\\
	S^{n+1}\xra{\eta_n}S^n\xra{i_n^\eta}C^{n+2}_\eta\xra{q_{n+2}^\eta}S^{n+2}.
\end{align*} 
For  a finitely generated abelian group $G$, we denote by $K_n(G)=K(G,n)$ the Eilenberg-MacLane space of homotopy type $(G,n)$. If $G$ contains $p$-torsion, $p$ a prime, we denote by $\rho_n\colon K_n(G)\to K_n(\Z/p)$ the homomorphism induced by the mod $p$ reduction $G\to \Z/p$. 
We write 
\[ G\cong C_1\langle x_1\rangle\oplus \cdots \oplus C_m\langle x_m\rangle\] 
if $G$ is the direct sum of the cyclic groups $C_i$ with generators $x_i$, $i=1,2,\cdots,m$. Denote by $(a,b)=\gcd\{a,b\}$ the greatest common divisor between two integers $a,b$.

We also need some technical lemmas. 
\begin{lemma}\label{lem:Chang}
	For $n\geq 3$, every $(n-1)$-connected complex of dimension at most $n+2$ with $2$-torsion-free homology is homotopy equivalent to a finite wedge sum of spheres $S^{n+i}$ for $i=0,1,2$, the mod $p^r$ Moore spaces $P^{n+j}(p^r)$ with $p$ odd numbers for $j=1,2$, and the $2$-cell Chang complex $C^{n+2}_\eta=\Sigma^{n-2}\C P^{2}$, where $\eta=\eta_n\colon S^{n+1}\to S^n$ is the iterated Hopf map. 
	\begin{proof}
	A special case of Chang's classification theorem \cite{Chang50}. 
	\end{proof}
\end{lemma}

\begin{lemma}[cf.  \cite{TodaBook}]\label{lem:Toda}
	The following hold:
	\begin{enumerate}
		\itemsep=5pt
		\item $\pi_3(S^2)\cong\Z\langle \eta\rangle$;  for $n\geq 3$, $\pi_{n+1}(S^n)\cong \z{}\lra{\eta_n}$, $\pi_{n+2}(S^n)\cong \z{}\lra{\eta^2=\eta_n\eta_{n+1}}$.
		\item $\pi_6(S^3)\cong\Z/12\langle \nu' \rangle$, the $3$-primary component of $\pi_6(S^3)$ is generated by $\alpha=4\nu'$; $\pi_7(S^4)\cong \Z\langle \nu_4\rangle\oplus\Z/12\langle \Sigma \nu'\rangle$,  $\pi_{n+3}(S^n)\cong\Z/{24}\langle \nu_n\rangle$ for $n\geq 5$. Note $\eta^3=6\nu'$, $2\nu_n=\Sigma^{n-3}\nu'$ for $n\geq 5$.
		\item $\pi_7(S^3)\cong \z{}\langle \nu'\eta_6=\eta\nu_4\rangle$, $\pi_8(S^4)\cong \z{}\langle \nu_4\eta_7\rangle\oplus\z{}\langle \eta_4\nu_5\rangle$, $\pi_9(S^5)\cong \z{}\lra{\nu_5\eta_8}$ and $\pi_{n+4}(S^n)=0$ for $n\geq 6$. Note $\eta\nu_n=0$ for $n\geq 6$.
	\end{enumerate}
\end{lemma}

\begin{lemma}[see \cite{HL-7mfd}]\label{lem:htps:Moore}
	Let $p$ be an odd prime and let $r\geq 1$.
	\begin{enumerate}
		\itemsep=2pt
		\item\label{Sn+2Pn+1} $\pi_{n+2}(P^{n+1}(p^r))\cong \pi^{n}(P^{n+2}(p^r))=0$ and 
		\[\pi_{n+2}(P^{n+1}(2^r))\left\{\begin{array}{ll}
			\Z/4\lra{\tilde{\eta}_1},&r=1;\\
			\z{}\lra{\tilde{\eta}_r}\oplus\z{}\lra{i_n\eta^2},&r\geq 2.
		\end{array}\right. \]
		Here $\tilde{\eta}_r$ satisfies the formulae
		\[q_{n+1}\tilde{\eta}_r=\eta,\quad 2\tilde{\eta}_1=i_n\eta^2.\]
		Dually, $[P^{n+2}(2^r),S^n]\cong\Z/4$ or $\z{}\oplus\z{}$ with a generator $\bar{\eta}_r$ satisfying 
		\[\bar{\eta}_ri_{n+1}=\eta,\quad 2\bar{\eta}_1=\eta^2 q_{n+2}.\]

		\item $\pi_8(P^6(3^r))\cong\Z/3\langle i_5\Sigma^2\alpha\rangle$ and $\pi_8(P^6(p^r))=0$ for $p\geq 5$.
		\item\label{htps:S7P4} $\pi_7(P^4(3^r))\cong\Z/3\lra{\Sigma \tilde{\alpha}_r}$ and $\pi_7(P^4(p^r))=0$ for $p\geq 5$, where $\tilde{\alpha}_r\in \pi_6(P^3(3^r))\cong\Z/3$ satisfies the formulae 
		\begin{equation*}
			q_3\tilde{\alpha}_r=\alpha,\quad B(\chi^r_s)\tilde{\alpha}_r=\tilde{\alpha}_s \text{ for }r\leq s.
		\end{equation*}
		Moreover, the suspension $E_\sharp\colon \pi_{n+3}(P^n(p^r))\to \pi_{n+4}(P^{n+1}(p^r))$ is an isomorphism for $n\geq 4$.

		\item\label{htps:S9P5} $\pi_9(P^5(p^r))=0$.
	\end{enumerate}
	
\end{lemma}

\begin{lemma}\label{lem:htpgrps}
	Let $p$ be an odd prime and let $r\geq 1,n\geq 3$ be integers.
	\begin{enumerate}
		\item\label{htpgrps:Ceta:stable} $\pi_n(C^{n+2}_\eta)\cong\Z\langle i_n^\eta\rangle$, $\pi_{n+2}(C^{n+2}_\eta)\cong\Z\langle \tilde{\zeta}\rangle$, $[P^{n+2}(p^r),C^{n+2}_\eta]\cong\zp{r}\lra{\tilde{\zeta}q_{n+2}}$, $\pi_{n+1}(C^{n+2}_\eta)=[P^{n+1}(p^r),C^{n+2}_\eta]=0$;\\
		dually, $\pi^{n+2}(C^{n+2}_\eta)\cong\Z\langle q_{n+2}^\eta\rangle$, $\pi^{n}(C^{n+2}_\eta)\cong\Z\langle \bar{\zeta}\rangle$, $[C^{n+2}_\eta,P^{n+1}(p^r)]\cong\zp{r}\langle i_n\bar{\zeta}\rangle$, $\pi^{n+1}(C^{n+2}_\eta)=[C^{n+2}_\eta,P^{n+2}(p^r)]=0$,  where $\tilde{\zeta}$ and $\bar{\zeta}$ satisfies the formulae
		\[q_{n+2}^\eta\tilde{\zeta}=2\cdot \id_{n+2},\quad \bar{\zeta}i_n^\eta=2\cdot \id_n, \quad \tilde{\zeta}q_{n+2}^\eta+i_n^\eta\bar{\zeta}=2\cdot\id_{\eta}.\]
	
		\item\label{htpgrps:Ceta-unst1} $\pi_6(C^5_\eta)\cong\Z/6\langle i_3^\eta\nu'=\Sigma \pi_2 \rangle$, $\pi_7(C^6_\eta)\cong\Z\langle i_4^\eta \nu_4 \rangle\oplus\Z/6\langle i_4^\eta\Sigma \nu'\rangle$ and $\pi_8(C^7_\eta)\cong\Z/12\langle i_5^\eta\nu_5\rangle$, where $\pi_2\colon S^5\to \C P^2$ is the canonical projection. Moreover, there hold formulae
		\begin{equation*}
			i_3^\eta\eta^3=0,\quad  \tilde{\zeta}\eta=3 i_3^\eta\nu'\in\pi_6(C^5_\eta).
		\end{equation*}
		\item\label{htpgrps:Ceta-unst2} $\pi_7(C^5_\eta)\cong\Z\lra{\xi}$, $\pi_8(C^6_\eta)\cong\z{}\lra{\Sigma \xi=i_4^\eta\nu_4 \eta_7}$ and $\pi_{n+4}(C^{n+2}_\eta)=0$ for $n\geq 5$.  
		\item\label{htpgrps:Ceta-unst3} $\pi_{10}(C^7_\eta)\cong\Z/24\lra{\tilde{\nu}_7}\oplus\z{}\lra{i_5^\eta\nu_5\eta^2}$ and $\pi_{11}(C^8_\eta)\cong \Z\lra{i_6^\eta[\id_6,\id_6]}\oplus\Z/24\lra{\Sigma\tilde{\nu}_7}$, where $\tilde{\nu}_7$ satisfies the formulae 
		\[q_7^\eta\tilde{\nu}_7=-\nu_7,\quad \tilde{\zeta}\nu_7=-2\tilde{\nu}_7+\epsilon\cdot i_5^\eta\nu_5\eta^2, ~\epsilon\in\{0,1\}.\]   
		Moreover, $\tilde{\nu}_7=\Sigma \tilde{\nu}_6$  for some $\tilde{\nu}_6\in \pi_9(C^6_\eta)$.

		
	\end{enumerate}
	\begin{proof}
		(1) The stable groups in  (\ref{htpgrps:Ceta:stable}) refer to \cite{Baues85} or \cite{lipc22}. (2) The groups $\pi_{6+i}(C^{5+i}_\eta)$ for $i=0,1,2$ are due to \cite[Proposition 8.2 and (13)]{Mukai82}; the formula $\tilde{\zeta}\eta=3 i_3^\eta\nu'$ is clear. 
		(3) The groups $\pi_{7+i}(C^{5+i}_\eta)$ are due to \cite[Proposition 8.3, 8.4 and Lemma 8.5]{Mukai82}.
		(4) By \cite[the top line of page 192 and Lemma 10.1]{Mukai82}, there is a split short exact sequence of abelian groups 
		\[\pi_{11}(S^7)=0\to \pi_{10}(S^5)\xra{(i_5^\eta)_\sharp}\pi_{10}(C^7_\eta)\xra{(q_7^\eta)_\sharp}\pi_{10}(S^7)\to 0.\]
		The group $\pi_{11}(C^8_\eta)$ refers to \cite[Proposition 10.2 (1)]{Mukai82}.
		
	For the ``moreover" statement, consider the EHP exact sequence 
		\[\pi_9(C^6_\eta)\xra{E_\sharp}\pi_{10}(C^7_\eta)\xra{H_\sharp}\pi_{10}(C^7_\eta\wedge C^6_\eta).\]
		We have $\pi_{10}(C^7_\eta\wedge C^6_\eta)\cong \pi_{10}(C^{11}_\eta\vee S^{11})=0$, and hence $E_\sharp$ is an epimorphism and $\tilde{\nu}_7$ is a suspension.
	\end{proof}
\end{lemma}

\section{Stable cohomotopy groups}\label{sec:chtp-stable}

Let $S^n_r$ denote the $(n+r)$-th Postnikov section of $S^n$, which is defined by the principal homotopy fibre sequence 
	\[K_{n+r}(\pi_{n+r}(S^n))\to S^n_r\xra{p_r}S^n_{r-1}\to K_{n+r+1}(\pi_{n+r}(S^n)).\]
Let $\imath_r\colon S^n\to S^n_r$ be the canonical $(n+r+1)$-connected map such that $q_r\imath_r=\imath_{r-1}$. We say that the Postnikov tower of $S^n$ is \emph{stable up to $n+s$ stages} if the groups $\pi_{n+r}(S^n)$ are stable for each $r\leq s$, under the Freudenthal suspension isomorphism theorem. 

\begin{lemma}\label{lem:tower:Sn}
	Up to $n+6$ stages, the stable Postnikov tower of $S^n$ takes the form:
	\begin{equation*}
		\begin{tikzcd}
			K_{n+6}(\Z/2)\ar[r,"j_{6}"]&S^n_{6}\ar[d,"p_{6}"]&\\
			K_{n+3}(\Z/24)\ar[r,"j_{3}"]&S^n_{3}\ar[d,"p_{3}"]\ar[r,"\ol{Sq^4_{\Z/8}}"]&K_{n+7}(\z{})\\
			K_{n+2}(\Z/2)\ar[r,"j_{2}"]&S^n_{2}\ar[r,"{(\wtd{Sq^4_\Z},\wtd{\PP_\Z})}"]\ar[d,"p_{2}"]&K_{n+4}(\Z/8)\times K_{n+4}(\Z/3)\\
			K_{n+1}(\Z/2)\ar[r,"j_{1}"]&S^n_{1}\ar[r,"\ol{Sq^2}"]\ar[d,"p_{1}"]&K_{n+3}(\z{})\\
			&K_n(\Z)\ar[r,"Sq^2_\Z"]&K_{n+2}(\z{})	
		\end{tikzcd}
	\end{equation*}
Here $Sq^i_G=Sq^i\circ \rho_2$ and $\PP_\Z=\PP_3\rho_3$;
$\ol{Sq^2}$ and $\ol{Sq^4_{\Z/8}}$ satisfy the formulae
\begin{equation}\label{eq:Sq2Sq4}
\ol{Sq^2}j_1=Sq^2,\quad  \ol{Sq^4_{\Z/8}}j_3=Sq^4_{\Z/8};
\end{equation}
$\wtd{Sq^4_\Z}$ and  $\wtd{\PP_\Z}$ satisfy the homotopy commutative diagrams 
\begin{equation}\label{wtd-Sq4-P1}
	\begin{tikzcd}
	S^n_2\ar[d,"p_1p_2"]\ar[r,"\wtd{Sq^4_\Z}"]&K_{n+4}(\Z/8)\ar[d,"\rho_2"]\\
	K_n(\Z)\ar[r,"Sq^4_\Z"]&K_{n+4}(\z{})
\end{tikzcd},\quad \begin{tikzcd}
	S^n_2\ar[d,"p_1p_2"]\ar[r,"\wtd{\PP_\Z}"]&K_{n+4}(\Z/3)\ar[d,equal]\\
	K_n(\Z)\ar[r,"\PP_\Z"]&K_{n+4}(\Z/3)
\end{tikzcd}.
\end{equation}
\begin{proof}
	See \cite[Figure 1, page 122]{MT68} or  \cite[Lemma 2.5]{GS21}.
\end{proof}
\end{lemma}

Note that in \cite{MT68} and \cite{GS21}, $\ol{Sq^2}$ and $\ol{Sq^4_{\Z/8}}$ are denoted by $\alpha$ and $P$, respectively; $\wtd{Sq^4_\Z}$ and $\wtd{\PP_\Z}$ are  denoted by ``$Sq^4\iota_n$'' and ``$Sq^4\iota_{n+3}$'', respectively.  In the paper, we use the above notation to avoid confusion: $\alpha$ is a generator of the $3$-primary component of $\pi_6(S^3)$, see  Lemma \ref{lem:Toda}; $P$ is usually used to denote the Peterson space $P^n(G)$ with the unique nontrivial reduced integral cohomology group $G$ in dimension $n$. 

For any abelian group $G$,  we denote by $QH^m(X,f)$ the cokernel of a homomorphism  $f\colon G\to H^m(X;\z{})$. 
Using the first nontrivial stage of the Postnikow tower in Lemma \ref{lem:tower:Sn}, we have the following characterization of the cohomotopy groups $\pi^{n}(X)$ of a complex of dimension $n+1$, where the splitting condition is due to Taylor \cite[Section 6.1]{Taylor2012}.

\begin{lemma}\label{lem:chtp=-1}
	Let $X$ be a CW complex of dimension at most $n+1$, $n\geq 3$. There is a short exact sequence of abelian groups
	\[0\to QH^{n+1}(X,Sq^2_\Z)\to \pi^n(X)\xra{h^n}H^n(X)\to 0,\]
	which splits if and only if $Sq^2_\Z\big(H^{n-1}(X;\Z)\big)=Sq^2\big(H^{n-1}(X;\z{})\big)$.
\end{lemma}

\begin{example}\label{ex:chtp=-1}
	Let $M$ be a closed orientable simply connected $(n+1)$-manifold, then 
	\begin{enumerate}
		\item $\pi^{n+1}(M)\cong\Z$ by the Hopf degree theorem, and $\pi^i(M)=0$ for $i=1$ or $i>n$;
		\item $\pi^{n}(M)\cong \z{1-\varepsilon}$, where $\varepsilon=0$ if $M$ is spin, otherwise $\varepsilon=1$.
	\end{enumerate}
\end{example}

We apply Lemma \ref{lem:tower:Sn} to analyze the stable chomotopy groups of manifolds as follows.

\begin{theorem}\label{thm:chtp-stable}
	Let $X$ be a CW complex such $\dim(X)\leq n+m$ and $H^n(X)$ is $2$-torsion-free for $n\geq 5$. 
	\begin{enumerate}[1.]
		\itemsep=2pt
		\item\label{chtp=-2} If $m=2$, $n\geq 4$ and $H^{n+1}(X;\z{})=0$, then there is a short exact sequence of abelian groups 
		\begin{equation}\label{SES:chtp=-2}
			0\to QH^{n+2}(X,\Omega \ol{Sq^2})\to \pi^n(X)\xra{\imath}\ker(Sq^2_\Z)\to 0,
		\end{equation}
		where $\imath$ is the  homomorphism such that the composition
		\[h^n\colon \pi^n(X)\xra{\imath}\ker(Sq^2_\Z)\hookrightarrow H^n(X)\]
		is the $n$-th cohomotopy Hurewicz homomorphism.
		
		The extension (\ref{SES:chtp=-2})  splits if $H^n(X)$ is $2$-torsion-free. 
		
		\item\label{chtp=-3} If  $H^{n+1}(X;\z{})=H^{n+2}(X;\z{})=0$, 
	 then in each case of (i) $m=3$ and $n\geq 5$, (ii) $m=4$ and $n\geq 6$, (iii) $m=5$ and $n\geq 7$, there is a short exact sequence of abelian groups
		\begin{equation}\label{SES:cohtp=-3}
			0\to G_{24} \to \pi^n(X)\xra{\jmath} \ker(\PP_\Z)\to 0,
		\end{equation}
		where $G_{24}$ is a subgroup of $H^{n+3}(M;\Z/24)$ and $\jmath$ is a homomorphism such that the composition 
		\[h^n\colon \pi^n(X)\xra{\jmath}\ker(\PP_\Z)\hookrightarrow H^n(X)\]
		is the $n$-th cohomotopy Hurewicz homomorphism.
		
		\item\label{chtp=-6} For $m=6$, $n\geq 8$,  if $H^{n+i}(X;\z{})=H^{n+3}(X;\Z/3)=0 \text{ for } i=1,2,3$,  then there is a split short exact sequence of abelian groups
		\begin{equation}\label{SES:chtp=-6}
			0\to QH^{n+6}(X,\Omega\ol{Sq^4_{\Z/8}})\to \pi^n(X)\to \ker(\PP_\Z)\to 0.
		\end{equation}

	\end{enumerate}
	
	\begin{proof}
	For any complex $X$, the stable Postnikov tower in Lemma \ref{lem:tower:Sn} induces the following four exact sequences of groups (denoted by $(\star)$)
		\begin{align*}
	0\to	QH^{n+1}(X,Sq^2_\Z)\to [X,S^n_1]\xra{(p_1)_\sharp}H^n(X)\xra{Sq^2_\Z} H^{n+2}(X;\z{}),\\
	0\to QH^{n+2}(X,\Omega \ol{Sq^2}) \to [X,S^n_2]\xra{(p_2)_\sharp}[X,S^n_1]\xra{(\ol{Sq^2})_\sharp} H^{n+3}(X;\z{}),\\
	 H^{n+3}(X;\Z/24)\to [X,S^n_3]\xra{(p_3)_\sharp}[X,S^n_2]\xra{(\wtd{Sq^4_\Z},\wtd{\PP_\Z})_\sharp}H^{n+4}(X;\Z/24),\\
		0\to QH^{n+6}(X,\Omega \ol{Sq^4_{\Z/8}})\to [X,S^n_6]\xra{(p_6)_\sharp}[X,S^n_3]\xra{(\ol{Sq^4_{\Z/8}})_\sharp}H^{n+7}(X;\z{}).
		\end{align*} 

 (1)	By the assumption of (\ref{chtp=-2}), the first two sequences of ($\star$) become
\begin{align*}
0\to [X,S^n_1]\xra{(p_1)_\sharp}H^n(X;\Z)\xra{Sq^2_\Z}H^{n+2}(X;\z{}),\\
0\to QH^{n+2}(X,\Omega \ol{Sq^2})\to 	[X, S^n_2]\xra{(p_2)_\sharp}[X,S^n_1]\to 0.
\end{align*}
Since the canonical map $\imath_2\colon S^n\to S^n_2$ is $(n+3)$-connected, we have an isomorphism $\pi^n(X)\cong [X, S^n_2]$ for $\dim(X)\leq n+2$.
Thus we get the short exact sequence (\ref{SES:chtp=-2}), which is clearly splitting if $H^n(X)$ is $2$-torsion-free.
	
(2) Under the assumptions, the first three exact sequences of $(\star)$ take the form	
	\begin{align*}
	0\to [X,S^n_1]\xra{(p_1)_\sharp}H^n(X)\to 0,\\
	0\to [X,S^n_2]\xra{(p_2)_\sharp}[X,S^n_1]\xra{(\ol{Sq^2})_\sharp}H^{n+3}(X;\z{}),\\
 0\to G_{24}\xra{(j_3)_\sharp}[X,S^n_3]\xra{(p_3)_\sharp}[X,S^n_2]\xra{(\wtd{Sq^4_\Z},\wtd{\PP_\Z})_\sharp}H^{n+4}(X;\Z/24),
\end{align*}
where $G_{24}=H^{n+3}(X;\Z/24)/(\Omega\wtd{Sq^4_\Z},\Omega\wtd{\PP_\Z})_\sharp(X,\Omega S^n_2)$.
It follows that there hold isomorphisms
\[	[X,S^n_1]\cong H^n(X),\quad	[X,S^n_2]\cong \ker((\ol{Sq^2})_\sharp)\cong H^n(X),\]
where the last isomorphism holds because $H^n(X)$ is $2$-torsion-free. Consequently, the homomorphism $(\wtd{Sq^4_\Z})_\sharp$ doesn't change the group structure of the kernel of $(\wtd{Sq^4_\Z},\wtd{\PP_\Z})_\sharp$; by (\ref{wtd-Sq4-P1}), we get 
\[\ker(\wtd{Sq^4_\Z},\wtd{\PP_\Z})_\sharp\cong\ker(\PP_\Z)\subseteq H^n(X).\]

Since $S^n_5=S^n_4=S^n_3$, we have the isomorphism 
\[(\imath_3)_\sharp\colon \pi^n(X)\to [X,S^n_3]\] 
in the three cases:
(i) $\dim(X)\leq n+3$, $n\geq 5$, (ii) $\dim(X)\leq n+4$, $n\geq 6$, (iii) $\dim(X)\leq n+5$, $n\geq 7$.  Thus from third exact sequence above, we deduce the short exact sequence of abelian groups
\[0\to G_{24}\to \pi^n(X)\xra{\jmath} \ker(\PP_\Z)\to 0.\]
In the case $\dim(X)\leq n+3$ with $n\geq 5$, $ \ker(\PP_\Z)=H^n(X)$ is $2$-torsion-free and $\jmath$ is the $n$-th cohomotopy Hurewicz homomorphism. 
The proof of (\ref{chtp=-3}) is finished.

(3)  Under the assumptions, we have the following exact sequences
	\begin{align*}
	0\to [X,S^n_1]\xra{(p_1)_\sharp}H^n(X)\to 0,\quad 
	0\to [X,S^n_2]\xra{(p_2)_\sharp}[X,S^n_1]\to 0,\\
	0\to [X,S^n_3]\xra{(p_3)_\sharp}[X,S^n_2]\xra{(\wtd{Sq^4_\Z}_\sharp,\wtd{\PP_\Z})}H^{n+4}(X;\Z/8)\oplus H^{n+4}(X;\Z/3),\\
0\to QH^{n+6}(X,\Omega\ol{Sq^4_{\Z/8}})\to [X,S^n_6]\xra{(p_6)_\sharp}[X,S^n_3]\to 0.
\end{align*}
Similarly, $H^n(X)$ is $2$-torsion-free implying that 
\[[X,S^n_3]\cong \ker(\PP_\Z)\subseteq H^n(X).\]
When $\dim(X)\leq n+6$ with $n\geq 8$, we have $\pi^n(X)\cong [X,S^n_6]$.
Then the extension (\ref{SES:chtp=-6}) follows by the last exact sequence, and the splitting is clear.
	\end{proof}
\end{theorem}

 As stated in the Introduction, the cohomotopy sets of connected manifolds of dimension at most $5$ are known \cite{Taylor2012,KMT2012,LZ-5mfld,ACS24}. We can apply Theorem \ref{thm:chtp-stable} to compute stable cohomotopy groups of $(n-1)$-connected $(2n+2)$-manifolds for $n\geq 2$.

\begin{corollary}\label{cor:chtp:2n+2}
Let $M$ be a closed orientable smooth $(n-1)$-connected $(2n+2)$-manifold such that $H_n(M)$ is $2$-torsion-free when $n\geq 3$. 
\begin{enumerate}
 \item  For $n\geq 2$, there is an isomorphism
 	\[\pi^{2n+1}(M)\cong QH^{2n+2}(M,Sq^2)\cong \z{1-\varepsilon},\]
	where $\varepsilon=0$ if $M$ is spin, otherwise $\varepsilon=1$.
 \item\label{chtp2n+2=-2} $\pi^{2n}(M)\cong \z{}$ for $n\geq 3$,  and there is a short exact sequence of abelian groups for $n=2$:
	\[ 0\to QH^{6}(M,Sq^2)\to \pi^{4}(M)\xra{\imath}\ker(Sq^2_\Z)\to 0,\]
	which splits if $H_2(M)$ is $2$-torsion-free. 
 \item\label{chtp2n+2=-3} For $n\geq 3$, there is a \text{split} short exact sequence of groups 
 \[0\to G_{24}\to \pi^{2n-1}(M)\xra{}H^{2n-1}(M)\to 0,\]
 where $G_{24}\subseteq H^{2n+2}(M;\Z/24)$. The splitting for $n\geq 4$ is clear and the splitting for $n=3$ follows by Proposition \ref{prop:cohtp=5}.

\item\label{chtp2n+2=-4} For $n\geq 4$, there is an isomorphism 
\[\pi^{2n-2}(M)\cong \ker\big(H^{2n-2}(M;\Z)\xra{\PP_\Z}H^{2n+2}(M;\Z/3)\big).\]
For $n\geq 5$, there is an isomorphism 
\[\pi^{2n-3}(M)\cong H^{2n-3}(M)\cong H_5(M).\]

\end{enumerate}

\end{corollary}
 \begin{proof}
	By Theorem \ref{thm:chtp-stable} it suffices to prove the
 isomorphism 
 \begin{equation}\label{eq:olSq2}
	QH^{2n+2}(M,\Omega\ol{Sq^2})\cong QH^{2n+2}(M,Sq^2).
 \end{equation}

	The formula $\ol{Sq^2}j_1=Sq^2$ and the snake lemma imply that there is a short exact sequence
\[0\to \frac{\im(\Omega\ol{Sq^2})_\sharp}{\im(Sq^2)} \to QH^{2n+2}(M;Sq^2)\to QH^{2n+2}(M,\Omega\ol{Sq^2})\to 0.\]
	It follows that if $Sq^2$ acts nontrivially on $H^{2n}(M;\z{})$, then \[\im(\Omega\ol{Sq^2})_\sharp=\im(Sq^2)\cong\z{},\quad QH^{2n+2}(M,\Omega\ol{Sq^2})=0;\]
	if $Sq^2$ acts trivially on $H^{2n}(M;\z{})$, then either (i)  or (ii) holds:
	\begin{align*}
		(i)~&~\im(\Omega\ol{Sq^2})_\sharp=0,~~ QH^{2n+2}(M,\Omega\ol{Sq^2})=\z{}; \\
		(ii)~&~\im(\Omega\ol{Sq^2})_\sharp\cong\z{},~~ QH^{2n+2}(M,\Omega\ol{Sq^2})=0.
	\end{align*}

	We check the isomorphism (\ref{eq:olSq2}) as follows.
	When $n\geq 4$, the $(2n-2)$-skeleton of $M$ is homotopy equivalent to the complement $\ol{M}$ of the top cell of $M$, hence we have 
	\[\pi^{2n}(M)\cong \pi^{2n}( M/\ol{M})\cong \pi^{2n+1}(S^{2n+3})\cong\z{},\]
	which implies the case $(i)$.
	
	When $n=3$, let $M_{(4)}$ be the $4$-skeleton of $M$. By (\ref{table:HM}), there is a homotopy cofibre sequence 
	\[S^7\xra{f}\bigvee_{i=1}^lS^5\to M/M_{(4)}\to S^8.\]
	Then it is easy to compute that 
	\[\pi^6(M)\cong \pi^6(M/M_{(4)})\cong \z{1-\epsilon},\]
	where $\epsilon=0$ if the secondary operation that detects $\eta^2$ acts trivially on $H^5(M;\z{})$, otherwise $\epsilon=1$. For smooth $2$-connected $8$-manifolds, we always have $\epsilon=0$ \cite[Lemma 4.3]{LZ-5mfld}, hence the case $(i)$ holds.
	 
	 When $n=2$, there is a homotopy cofibre sequence by (\ref{table:HM}) 
	 \[S^6\xra{g}\bigvee_{i=1}^kS^4\vee \bigvee_{i=1}^lS^5\vee P^5(T)\to \Sigma M/\Sigma M_{(2)}\to S^7,\]
	 where $M_{(2)}$ is the $2$-skeleton of $M$ and $g$ is a map contains no Whitehead products. By Lemma \ref{lem:htpgrps}, we can set 
	 \[g=\sum_{i=1}^kx_i\cdot \eta^2+\sum_{i=1}^ly_i\cdot \eta+\sum_{i=1}^t(z_i^1\cdot \tilde{\eta}_{r_i}+z_i^2\cdot i_4\eta^2),\]
	 where all coefficients belong to $\{0,1\}$.
	 When $M$ is spin, the Steenrod square $Sq^2$ acts trivially on $H^5(\Sigma M;\z{})$, implying that $y_i=z_i^1=0$ for all $i$; by \cite[Lemma 4.3]{LZ-5mfld}, the secondary operation that detects $\eta^2$  acts trivially on $H^4(\Sigma M;\z{})$, implying that $x_i=z_i^2=0$ for all $i$.  Hence the spin and smooth condition implies that $g$ is null homotopic, yielding a homotopy equivalence
	 \[\Sigma M/\Sigma M_{(2)}\simeq \bigvee_{i=1}^lS^5\vee \bigvee_{i=1}^kS^4\vee P^5(T)\vee S^6.\]
Thus we have  $\pi^4(M)\cong \pi^5(\Sigma M/\Sigma M_{(2)})\cong H^4(M)\oplus \z{},$
	 which implies the case $(i)$.
 \end{proof}

\section{Unstable cohomotopy of 6-, 8-, 10-manifolds}\label{sec:chtp=2n+2}

In this section we shall study unstable cohomotopy sets of $(n-1)$-connected $(2n+2)$-manifolds $M$ for $n=2,3,4$ in terms of the homotopy type of the reduced suspension space $\Sigma M$ given by Propositions \ref{prop:SM-hbar}, \ref{prop:SM-8mfd-phi} and \ref{prop:SM-10mfd}, whose proofs are deferred to Section \ref{sec:SM-proofs}.

\begin{proposition}\label{prop:SM-hbar}
	Let $M$ be a simply connected $6$-manifold with $H_\ast(M)$ given by (\ref{table:HM}), $n=2$. If $T$ is  $2$-torsion-free, then there exist integers $0\leq c\leq l$ and $0\leq r_{j_0}<\infty$ that depend on $M$ such that 
	\begin{equation*}
		\Sigma M\simeq \bigvee_{i=2}^{l-c}(S^3\vee S^5)\vee \bigvee_{i=1}^kS^4\vee P^4\big(\frac{T}{\Z/3^{r_{j_0}}}\big)\vee P^5(T)\vee \bigvee_{i=2}^{c}C^5_\eta\vee C_\hbar, 
	\end{equation*}
	where $c=0$ if and only if $Sq^2$ acts trivially on $H^2(M;\z{})$, $r_{j_0}=0$ if $T$ contains no $3$-torsion, and $\hbar\colon S^6\to S^3\vee S^5\vee P^4(3^{r_{j_0}})\vee C^5_\eta$ is some map given by the expression
	\begin{equation}\label{eq:hbar}
		\hbar=x\cdot \nu'+y\cdot \eta_5+z_{j_0}\cdot i_3\alpha+w\cdot i_3^\eta\nu',
	\end{equation}
	where $x\in\{0,\pm 1,\pm 2,\pm 3,\pm 4,\pm 5,6\}$,  $z_{j_0}\in\{0,1\}$,  $w\in \{0,\pm 1,\pm 2,3\}$, and $y=0$ if $M$ is spin, otherwise $y=1$.
	
\end{proposition}

\begin{proposition}\label{prop:SM-8mfd-phi}
	Let $M$ be a closed orientable smooth $2$-connected $8$-manifold with $H_\ast(M)$ given by (\ref{table:HM}), $n=3$. If $T$ is $2$-torsion-free, then there is a homotopy equivalence 
	\begin{multline*}
		\Sigma M\simeq \bigvee_{i=2}^k S^5\vee \bigvee_{i=2}^{l-c}S^{4} \vee \bigvee_{i=1}^{l-c} S^6\vee \bigvee_{i=2}^c C_{\eta}^{6} \vee P^{5}(\frac{T}{\Z/3^{r_{j_0}}}) \vee P^{6}(\frac{T}{\Z/3^{r_{j_1}}})\vee  C_{\phi},
	\end{multline*}
	where $0\leq c\leq l$, $r_{j_0},r_{j_1}$ are non-negative integers that depend on $M$ and $\phi\colon S^8\to S^5\vee S^4\vee C^6_\eta\vee P^5(3^{r_{j_0}})\vee P^6(3^{r_{j_1}})$ is given by 
	\begin{equation}\label{eq:phi}
		\phi=x\cdot \nu_5+y\cdot \eta_4\nu_5+z\cdot i_4^\eta\nu_4\eta_7 +u_{j_0}\cdot \Sigma^2\tilde{\alpha}_{r_{j_0}}+w_{j_1}\cdot i_5\Sigma^2\alpha,
	\end{equation} 
	where $x\in\Z/24=\{0,\pm 1,\pm 2,\cdots,\pm 11,12\}$, $y,z\in\{0,1\}$ and $u_{j_0},w_{j_1}\in\{0,\pm 1\}$. 
	
\end{proposition}

\begin{proposition}\label{prop:SM-10mfd}
	Let $M$ be a closed orientable $3$-connected $10$-manifold with $H_\ast(M)$ given by (\ref{table:HM}), $n=4$. If $T$ is $2$-torsion-free, then there is a homotopy equivalence 
	\begin{multline*}
		\Sigma M\simeq \bigvee_{i=1}^k S^6\vee \bigvee_{i=2}^{l-c}(S^{5}\vee S^7) \vee \bigvee_{i=3}^c C^7_{\eta} \vee P^{6}(T)\vee P^{7}(\frac{T}{\Z/3^{r_{j_0}}}) \vee C_{\varphi},
	\end{multline*}
	where  $0\leq c\leq l$ and $r_{j_0}$ are non-negative integers that depend on $M$, and $\varphi\colon S^{10}\to S^7\vee S^5\vee C^7_\eta\vee C^7_\eta\vee P^7(3^{r_{j_0}})$ is given by 
	\begin{equation}\label{eq:varphi}
		\varphi=x\cdot\nu_7+y\cdot \nu_5\eta^2+(z\cdot \Sigma\tilde{\nu}_6+\epsilon\cdot i_5^\eta\nu_5\eta^2)+\delta\cdot i_5^\eta\nu_5\eta^2+w_{j_0}\cdot \Sigma^4\tilde{\alpha}_{r_{j_0}},
	\end{equation} 
	where $x,z\in\Z/24$, $w_{j_0}\in\Z/3=\{0,\pm 1\}$ and  
	\[(y,\epsilon,\delta)\in \{(0,0,0),(1,0,0),(0,1,0),(0,0,1)\}.\] 
\end{proposition}

\subsection{Cohomotopy of simply connected $6$-manifolds}\label{sec:chtp:n=2}
Let $M$ be a simply connected $6$-manifold with $H_\ast(M)$ given by (\ref{table:HM}), $n=2$. 
Let $ BX$ be the classifying space of a homotopy associative $H$-space. It is well-known that $ BS^1=K(\Z,2)=\C P^{\infty}$ and $ BS^3=\HP{\infty}$. Since $\pi^1(M)=H^1(M)=0$, the complex Hopf fibration 
$S^1\to S^3\xra{\eta}S^2$ induces a short exact sequence of sets
\[0\to \pi^3(M)\xra{}\pi^2(M)\xra{h^2}H^2(M)\xra{\omega_\sharp}[M,B S^3],\] 
where $\omega\colon BS^1\to BS^3$ is the canonical map of $\C P^{\infty}\to \HP{\infty}$.
The natural action of $\pi^3(M)$ on $\pi^2(M)$ is transitive, and by \cite[Theorem 3]{KMT2012}, $\pi^1(M)=0$ implying that the action is also free. Consequently, $\pi^2(M)$ is a left $\pi^3(M)$-torsor, and hence there is a bijection \[\pi^3(M)\xra{1:1}\pi^2(M).\]

To understand the cohomotopy Hurewicz map $h^2\colon \pi^2(M)\to H^2(M)$, we shall employ appropriate cohomology operations to provide an equivalent  description of the condition $\omega_\sharp(u)=0$ for a class $u\in H^2(M)$. The analysis method of the following Proposition \ref{prop:chtp=2} is due to \cite[Section 5]{Taylor2012} or \cite[Proposition 6.10]{ACS24}.

\begin{proposition}\label{prop:chtp=2}
   Let $M$ be a simply connected $6$-manifold with $H_\ast(M)$ given by (\ref{table:HM}), $n=2$.
\begin{enumerate}
   \item\label{chtp=2:nonspin} If $M$ is nonspin,  there is an exact sequence of sets 
	\[0\to\pi^3(M)\xra{}\pi^2(M)\xra{h^2}H^2(M)\xra{\smallsmile^2}H^4(M),\]
	where $\smallsmile^2$ is the cup square operation sending $u$ to $u^2$.
   \item\label{chtp=2:spin} If $M$ is spin, then there is an exact sequence of sets 
   \[0\to\pi^3(M)\xra{}\pi^2(M)\xra{h^2}H^2(M)\xra{(\smallsmile^2,\Theta_0)}H^4(M)\oplus \z{},\]
   where $\Theta_0\colon \{u\in H^2(M;\z{})\mid u^2=0\} \to H^6(M;\z{})$
   is the restriction of the secondary cohomology operation $\Theta$ constructed by the diagram (\ref{diag:Theta}).
\end{enumerate}

\begin{proof}
	 Let $(BS^3)_n$ be the $n$-th Postnikov section of $ BS^3$.
Since $\iota\colon S^4\to  BS^3$ is $7$-connected, $ BS^3$ has the same Postnikov tower as $S^4$ up to stage $6$, see Lemma \ref{lem:tower:Sn} with $n=4$. Note that we have a bijection $[M, BS^3]\cong [M,S^4_2]$. 

(1) When $M$ is nonspin, by Lemma \ref{lem:tower:Sn} and Corollary \ref{cor:chtp:2n+2}, there is a sequence of isomorphisms 
\[\pi^4(M)\xra[\cong]{(\imath_2)_\sharp}[M,S^4_2]\xra[\cong]{(p_2)_\sharp} [M,S^4_1]\xra[\cong]{(p_1)_\sharp}\ker(Sq^2_\Z).\] 
It follows that $0=\omega_\sharp(u)\in\pi^4(M)$ if and only if 
\[0=(p_1)_\sharp(p_2)_\sharp((\imath_2)_\sharp)\omega_\sharp(u)=i_\sharp\omega_\sharp(u)=u^2,\]
which completes the proof of (\ref{chtp=2:nonspin}).

(2) When $M$ is spin, by Corollary \ref{cor:chtp:2n+2} with $n=2$,  there is a short exact sequence
\begin{equation}\label{SES:MS4}
	0\to \z{}\to \pi^4(M)\xra{h^4}H^4(M)\to 0,
\end{equation}
where $h^4$ is the fourth cohomotopy Hurewicz homomorphism and satisfies the formula $h^4\circ w_\sharp(u)=u^2$. 

To analyze the condition $\omega_\sharp(u)=0$, we don't need the full structure of $\pi^4(M)$, hence we replace it with $H^4(M)\oplus \z{}$. 
Consider the following homotopy commutative diagram associated to the Postnikov tower of $B S^3$ or $S^4$:
\begin{equation}\label{diag:E5E6}
   \begin{tikzcd}
	   &	K(\Z/2,6)\ar[r,"j_2"]&S^4_2\ar[d,"p_2"]&\\
	   &K(\Z/2,5)\ar[r,"j_1", near start]&S^4_1\ar[r,"\ol{Sq^2}"]\ar[d,"p_1"]&K(\z{},7)\\
	   M\ar[uur,"u''"]\ar[ur,"u'\simeq \ast"description]\ar[r,"u"]&\C P^{\infty}\ar[uur,"g_2"description, crossing over, near end]\ar[r,"\smallsmile^2"]\ar[ur,"g_1"description]&K(\Z,4)\ar[r,"Sq^2_\Z"]&K(\z{},6)	
   \end{tikzcd},
\end{equation}
where $\smallsmile^2$ is the cup square operation. Let $u\in H^2(M)$ be a class satisfying $u^2=0$, or equivalently $\smallsmile^2\circ u\simeq\ast$.  Since $Sq^2_\Z\circ \smallsmile^2\simeq\ast$, there is a lift $g_1$ of $\smallsmile^2$. $H^7(\C P^{\infty};\z{})=0$ implying a lift $g_2$ of $g_1$ such that $p_2g_2=g_1$. Then $p_1g_1u=u^2=0$ implying that there exists a lift $u'\colon M\to K(\z{},5)$ such that $j_1u'\simeq g_1u$. Note that $M$ is simply connected implying that $u'$ is null homotopic  and $(p_1)_\sharp\colon [M,S^4_1]\xra{}H^4(M)$ is injective; consequently, $(p_1)_\sharp(g_1u)=u^2=0$ implying $g_1u$ is null homotopic. Then
$p_2g_2u=g_1u\simeq \ast$ implying a lift $u''\colon M\to K(\Z/2,6)$ such that 
\begin{equation}\label{eq:u''}
	g_2u\simeq j_2 u''.
\end{equation}

Consider the following homotopy commutative diagram induced by the null homotopy $\ol{Sq^2}g_1\simeq\ast$:
\begin{equation}\label{diag:Theta}
   \begin{tikzcd}
	   &&M\ar[d,"g_1u\simeq\ast"]\ar[dl,"u"swap]\ar[dll,"\tilde{u}"swap]\\
	   F\ar[r]\ar[d,"\widetilde{Sq^2}"]&K(\Z,2)\ar[r,"g_1"]\ar[d,"g_2"]&S^4_1\ar[d,equal]\\
	   K(\z{},6)\ar[r,"j_2"]&S^4_2\ar[r,"p_6"]&S^4_1\ar[r,"\ol{Sq^2}"]&K(\z{},7)
   \end{tikzcd},
\end{equation}
where $F$ is the homotopy fibre of $g_1$,  $\tilde{u}$ and $\widetilde{Sq^2}$ are the induced lifting maps. By \cite{Harperbook}, there is a secondary cohomology operation $\Theta$ (based on the relation $\ol{Sq^2}g_1\simeq\ast$) defined by 
\[\Theta\colon \ker((g_1)_\sharp)\to H^6(M;\z{}),\quad \Theta(u)=\widetilde{Sq^2}\circ \tilde{u},\]
where the indeterminacy vanishes because $\im(\Omega\ol{Sq^2})_\sharp=\im(Sq^2)=0$, by the short exact sequence (\ref{SES:MS4}) and the formula (\ref{eq:olSq2}). This fact, together with (\ref{eq:u''}), implies that 
\[u''=\widetilde{Sq^2}\circ \tilde{u}=\Theta(u).\] 
Let $\Theta_0$ be the restriction of $\Theta$ to the subgroup $\{u\in H^2(M;\z{})\mid u^2=0\}$ of $\ker((g_1)_\sharp)$. Then we have $u''=\Theta_0(u)$ by comparing the diagrams (\ref{diag:E5E6}) and (\ref{diag:Theta}). 
\end{proof}
\end{proposition}

The quaternionic multiplication on $S^3$ gives a homotopy associative $H$-structure on $S^3$, yielding the standard group structure on the cohomotopy set $\pi^3(X)$. Since $S^3=\Omega BS^3$, where $BS^3$ is the classfying space of $S^3$, the group structure on $\pi^3(M)$ can also be induced by the standard comultiplication of the suspension $\Sigma M$ under the bijection $\pi^3(M)=[\Sigma M,BS^3]$.  Applying the discussion of the homotopy decomposition of the reduced suspension space $\Sigma M$ in Section \ref{sec:SM-proofs}, we  characterize the group structure on $\pi^3(M)$ by the following proposition.

\begin{proposition}\label{prop:chtp=3}
	Let $M$ be a simply connected $6$-manifold with $H_\ast(M)$ given by (\ref{table:HM}), $n=2$. Let $c$ be given by Proposition \ref{prop:SM-hbar}. If $T$ is $2$-torsion-free, then there is a short exact sequence of groups 
	\[0\to G_{12}\oplus T\to \pi^3(M)\to \Z^k\oplus \bigoplus_{i=1}^{l-c-\varepsilon}\z{}\to 0,\]
	where $G_{12}$ is a subgroup of $\Z/12$, and $\varepsilon=0$ if $M$ is spin, otherwise $\varepsilon=1$. Moreover, the homomorphism $\pi^3(M)\to \bigoplus_{i=1}^{l-c-\varepsilon}\z{}$ admits a splitting section, and the action of $ \Z^k\oplus \bigoplus_{i=1}^{l-c-\varepsilon}\z{}$ on $T$ is trivial. 
	
	In particular, if $k=l=0$, then there is an isomorphism 
	\[\pi^3(M)\cong T\oplus \Z/12.\]
	
\begin{proof}
	Let $h\colon S^5\to \ol{M}$ be the attaching map of   the top cell of $M$, where $\ol{M}$ denotes $5$-skeleton of $M$. Then the homotopy cofibre sequence $$S^6\xra{\Sigma h}\Sigma \ol{M}\xra{\Sigma i}\Sigma M\xra{\Sigma q} S^7$$ induces a long exact sequence of groups 
	\[[\Sigma^2 \ol{M},BS^3]\xra{(\Sigma^2 h)^\sharp}\pi_7(BS^3)\xra{(\Sigma q)^\sharp}[\Sigma M,BS^3]\xra{(\Sigma i)^\sharp}[\Sigma \ol{M},BS^3]\xra{(\Sigma h)^\sharp}\pi_6(BS^3).\]
	The cokernel of $(\Sigma^2 h)^\sharp$ is isomorphic to some quotient group $G_{12}$ of $\pi_7(BS^3)\cong \Z/12$.
	By Proposition 4.1, there is a homotopy equivalence 
	\[\Sigma \ol{M}\simeq \bigvee_{i=1}^k S^4\vee \bigvee_{i=1}^{l-c}(S^3\vee S^{5})\vee \big(\bigvee_{i=1}^c C^{5}_\eta\big)\vee P^{4}(T)\vee P^{5}(T). \]
  Note that $[\Sigma \ol{M},BS^3]$ is stable under the suspension,  one can easily compute that there is an isomorphism 
	\begin{align*}
		[\Sigma \ol{M},BS^3]&\cong [\bigvee_{i=1}^kS^4,BS^3]\oplus [P^4(T),BS^3] \oplus [\bigvee_{i=1}^{l-c}S^5,BS^3]\\
		&\cong \Z^k\oplus T\oplus \bigoplus_{i=1}^{l-c}\z{}.
	\end{align*}
	Since $\pi_6(BS^3)\cong \Z/2$, the group $\Z^k\oplus T$ clearly belongs to the kernel of $(\Sigma h)^\sharp$. 
	The composition 
	\[S^6\xra{\Sigma h}\Sigma \ol{M}\xra{\mathrm{pinch}}\bigvee_{i=1}^{l-c}S^5\]
	is given by the sum $\sum_{i=1}^{l-c}\varepsilon_i\cdot \eta$, where $\varepsilon_i\in\{0,1\}$ and all $\varepsilon_i$ are $0$ if and only if $M$ is spin. 
It follows that there is a short exact sequence of groups 
\[0\to G_{12}\to [\Sigma M,BS^3]\xra{(\Sigma i)^\sharp}\Z^k\oplus T\oplus \bigoplus_{i=1}^{l-c-\varepsilon}\z{}\to 0,\tag{SES-1}\label{ses1}\]
where $\varepsilon=0$ if $M$ is spin, otherwise $\varepsilon=1$.

Write $T=T_3\oplus T_{\geq 5}$, where $T_3\cong \bigoplus_{j=1}^t \Z/3^{r_j}$ and $T_{\geq 5}$ consists of $p$-torsion elements of $T$ for $p\geq 5$. It is clear that $T_{\geq 5}$ is a direct summand of $[\Sigma M,BS^3]$. The direct sum
$\bigoplus_{i=1}^{l-c-\varepsilon}\z{}$ corresponds to the wedge summand $\bigvee_{i=1}^{l-c-\varepsilon}S^5$ in the homotopy decomposition of $\Sigma M$ given by Proposition \ref{prop:SM-hbar}. Recall that the co-H-deviation of a map $f\colon \Sigma M\to Y$ between co-H-spaces is the difference between the compositions $\mu_Y \circ f$ and $(f\vee f)\circ \mu_{\Sigma M}$, which lifts to some homotopically unique map $\delta(f)\colon\Sigma M\to\Sigma \Omega Y\wedge \Omega Y$. Here $\mu_Y$ and $\mu_{\Sigma M}$ denote the given comultiplications on $Y$ and $\Sigma M$, respectively. Then by dimension and connectivity reasons, the pinch map $\tau_S\colon \Sigma M\to Y=\bigvee_{i=1}^{l-c-\varepsilon}S^5$ is a co-H-map and hence the induced map 
\[\tau_S^\sharp\colon \bigoplus_{i=1}^{l-c-\varepsilon}\pi_5(BS^3)\to [\Sigma M,BS^3]\] 
is a homomorphism. By Proposition \ref{prop:SM-hbar}, $\tau_S^\sharp$ is a splitting section of the homomorphism $[\Sigma M,BS^3]\xra{(\Sigma i)^\sharp}[\Sigma \ol{M},BS^3]\to \bigoplus_{i=1}^{l-c-\varepsilon}\pi_5(BS^3)$. In other words, the direct sum $\bigoplus_{i=1}^{l-c-\varepsilon}\z{}$ splits off $[\Sigma M,BS^3]$.

Define $h_1$ by the composition $$h_1\colon S^6\xra{\Sigma h} \Sigma \ol{M}\simeq \bigvee_{i=1}^k S^4\vee P^4(T) \vee \bigvee_{i=1}^{l-c} S^{5}\vee Z,$$ where $Z\simeq P^5(T)\vee \bigvee_{i=1}^{l-c}S^3\vee \bigvee_{i=1}^{c} C^5_\eta$.
By the proof of Proposition \ref{prop:SM-hbar} given in Section \ref{sec:SM-proofs}, the map $h_1$ factors as the composition 
\[h_1\colon S^6\xra{\hbar}S^3\vee S^5\vee P^4(3^{r_{j_0}})\vee C^5_\eta\hookrightarrow \bigvee_{i=1}^k S^4\vee P^4(T) \vee\bigvee_{i=1}^{l-c} S^{5}\vee Z,\]
where $\hbar=x\cdot \nu'+y\cdot \eta_5+z_{j_0}\cdot i_3\alpha+w\cdot i_3^\eta\nu'$. Let 
\[\hbar_b=x\cdot \nu'+y\cdot \eta_5+z_{j_0}\cdot\alpha+w\cdot i_3^\eta\nu'\colon S^6\to S^3\vee S^5\vee S^3\vee C^5_\eta\]
and define $h_b$ by the composition 
\[h_b\colon S^6\xra{\hbar_b} S^3\vee S^5\vee S^3\vee C^5_\eta \hookrightarrow \bigvee_{i=1}^k S^4\vee \bigvee_{j=1}^t S^3 \vee\bigvee_{i=1}^{l-c} S^{5}\vee Z.\] 
Consider the following homotopy commutative diagram of homotopy cofibre sequences
\[\begin{tikzcd}
	S^6\ar[r,"h_1"]&\bigvee_{i=1}^k S^4\vee P^4(T)\vee \bigvee_{i=1}^{l-c} S^{5}\vee Z\ar[r]&C_{h_1}\\
	S^6\vee \bigvee_{j} S^3\ar[r,"h_2"]\ar[u,"p_1"]&\bigvee_{i=1}^k S^4\vee \bigvee_j S^3\vee \bigvee_{i=1}^{l-c} S^{5}\vee Z\ar[r]\ar[u,"\id\vee (\bigvee_j i_3)\vee \id\vee \id"]&C_{h_2}\ar[u,"\lambda_b"]
\end{tikzcd}\]
that induces the map $\lambda_b$,
where $h_2|S^6=h_b$ and $h_2|S^3$ is the composition \[\bigvee_jS^3\xra{\bigvee_j 3^{r_j}}\bigvee_{j=1}^t S^3\hookrightarrow \bigvee_{i=1}^k S^4\vee \bigvee_{j=1}^t S^3\vee \bigvee_{i=1}^{l-c} S^{5}\vee Z.\] 
Then it is easy to see that the map $\lambda\colon C_{h_2}\to C_{h_1}$ induces an isomorphism of integral homology groups and hence is a homotopy equivalence by the Whitehead theorem. 
By similar computation and analysis, the homotopy cofibre sequence for $C_{h_2}$ in the above diagram induces a split short exact sequence of groups
\[
	0\to G_{12}\oplus T\to [C_{h_2},BS^3]\to \Z^k\oplus \bigoplus_{i=1}^{l-c-\varepsilon}\z{}\to 0.\tag{SES-2}\label{ses2}
\]

Comparing these two short exact sequences (\ref{ses1}) and (\ref{ses2}), we see that the group $T$ acts trivially on $G_{12}$ in (\ref{ses1}), the action of $\Z^k\oplus \bigoplus_{i=1}^{l-c-\varepsilon}\z{}$ on $T$ in (\ref{ses2}) is trivial, and that the action of $\Z^k\oplus \bigoplus_{i=1}^{l-c-\varepsilon}\z{}$ on $G_{12}$ in (\ref{ses1}) coincides with that in  (\ref{ses2}). 
When $k=l=0$, we compute that $G_{12}\cong \Z/12$, and hence the second exact sequence yields an isomorphism $\pi^3(M)\cong T\oplus \Z/12$.
\end{proof}
\end{proposition}

\begin{remark}\label{rmk:chtp=3}
By Proposition \ref{prop:SM-hbar} or the homotopy equivalence (\ref{eq:SM-SM'}), there is a retraction $\varrho_S\colon \Sigma M\to \bigvee_{i=1}^{k/2}(S^4\vee S^4)$ such that $\varrho_S\circ \Sigma j_S\simeq \id$, where $j_S\colon Y=\bigvee_{i=1}^{k/2}(S^3\vee S^3)\to M$ is the canonical inclusion map. Since $\Sigma M$ has dimension $7$, the co-H-deviation of $\varrho_S$ factors as the composition 
\[D\colon \Sigma M\xra{\Sigma q}S^7\xra{d}\bigvee_{i=1}^k\Sigma S^3\wedge S^3\xra{\mathrm{incl}}\Sigma Y\wedge Y\xra{}\Sigma Y\vee \Sigma Y,  \]
where the last map is the fibre inclusion map of the canonical inclusion $\Sigma Y\vee \Sigma Y\to \Sigma Y\times \Sigma Y$. Note that the composition of last two maps is a sum $\sum_{i,j}[\iota_i,\iota_j]$ of Whitehead products of the form $[\iota_i,\iota_j]$, where $\iota_i$ and $\iota_j$ are inclusions of $S^4$'s into the wedge $Y$ of  $S^4$'s. Consider the exact sequence of groups
\[[\Sigma^2\ol{M},\Sigma Y\vee \Sigma Y]\xra{(\Sigma h_1)^\sharp} \pi_7(\Sigma Y\vee \Sigma Y)\xra{(\Sigma q)^\sharp}[\Sigma M,\Sigma Y\vee \Sigma Y].\]
If the co-H-deviation $D$ is null-homotopic, then $D'=\big(\sum_{i,j}[\iota_i,\iota_j]\big)d$ belongs to the image of $(\Sigma h_1)^\sharp$. By Proposition \ref{prop:SM-hbar} and some computation, we have $D'=x\cdot \Sigma \nu'$ for some $x\in\Z/12$.
Since $\Sigma D'=0$, $\Sigma^2\nu'=2\nu_5$, we must have $x=0$. Thus the co-H-deviation $D$ is null-homotopic if and only if $D'$ is null-homotopic, which is equivalent to the condition that all cup products in the cohomology ring of the homotopy cofibre of $D'$ are zero, by similar arguments to the proof of \cite[Lemma 2.8]{LZ-5mfld}. There seems to not be enough information to derive the last statement, and we cannot conclude that $\varrho_S$ is a co-H-map. Consequently, the induced map 
\[\varrho_S^\sharp\colon \Z^k\to [\Sigma M,BS^3]=\pi^3(M)\]
may not be a homomorphism.
\end{remark}


\begin{proof}[Proof of Theorem \ref{thm:chtp-6mfd}]
Example \ref{ex:chtp=-1} implies the statement (\ref{chtp=5}), 
Corollary \ref{cor:chtp:2n+2} (\ref{chtp2n+2=-2}) with $n=2$ describes $\pi^i(M)$ for $i=4,5$, Proposition \ref{prop:chtp=2} enumerates $\pi^2(M)$ and Proposition \ref{prop:chtp=3} characterizes $\pi^3(M)$, respectively.
\end{proof}

\subsection{Cohomotopy of $2$-connected $8$-manifolds}\label{sec:chtp:n=3}

In this subsection we apply Proposition \ref{prop:SM-8mfd-phi} to  compute the cohomotopy groups $\pi^5(M)$ and $\pi^3(M)$ of a closed orientable smooth $2$-connected $8$-manifold $M$. Recall that there is a homotopy cofibre sequence 
\begin{equation*}
	S^8\xra{\phi} V=S^5\vee S^4\vee C^6_\eta\vee P^5(3^{r_{j_0}})\vee P^6(3^{r_{j_1}})\xra{i}C_{\phi}\xra{q}S^9,
\end{equation*}
where $\phi$ has the expression (\ref{eq:phi}).
Denote by $\iota \colon S^4\to  BS^3$ the bottom inclusion map.

\begin{lemma}\label{lem:Cphi-S}
There is an isomorphism	$[C_{\phi},S^6]\cong \Z\oplus\Z/3^{r_{j_1}}\oplus G_{24}$, where $G_{24}$ is a subgroup of $\Z/24$.

	\begin{proof}
		Let $W_b=S^5\vee S^4\vee C^6_\eta\vee P^5(3^{r_{j_0}})\vee S^5$ and consider the following homotopy commutative diagram that induces the map $\lambda$:
		\[\begin{tikzcd}
			S^8\ar[r,"\phi"]& W\ar[r,"i"]&C_\phi\ar[r,"q"]&S^9\\
			S^8\vee S^5\ar[u,"p_1"]\ar[r,"\phi_b"]&W_b\ar[u,"1\vee 1\vee 1\vee 1\vee i_5"]\ar[r,"i_b"]&C_{\phi_b}\ar[u,"\lambda"]\ar[r,"q_b"]&S^9\vee S^6\ar[u,"p_1"]
		\end{tikzcd},\]
		where 
		\begin{align*}
			\phi_b|S^8&=x\cdot \nu_5+y\cdot \eta_4\nu_5+z\cdot i_4^\eta\nu_4\eta_7 +u_{j_0}\cdot \Sigma^2\tilde{\alpha}_{r_{j_0}}+w_{j_1}\cdot \Sigma^2\alpha,\\
			\phi_b|S^5&\colon S^5\xra{3^{r_{j_1}}}S^5\xra{\imath_5}S^5\vee S^4\vee C^6_\eta\vee P^5(3^{r_{j_0}})\vee S^5.
		\end{align*}
		Then one checks that $\lambda$ is a homology equivalence and hence a homotopy equivalence by the Whitehead theorem. 
		
		Consider the induced exact sequence of abelian groups
		\[ [\Sigma W_b,S^6]\xra{\Sigma \phi_b^\sharp}[S^9\vee S^6,S^6]\xra{q_b^\sharp}[C_{\phi},S^6]\xra{i_b^\sharp} [W_b,S^6]\xra{\phi_b^\sharp}[S^8\vee S^5,S^6].\]
		We compute that 
		\begin{align*}
			\cok(\Sigma\phi_b^\sharp)&\cong\frac{\Z/24}{(24,x+8u_{j_0}+8w_{j_1})}\oplus \Z/3^{r_{j_1}},\\
			\ker(\phi_b^\sharp)&=[W_b,S^6]=[C^6_\eta,S^6]\cong\Z\lra{q_6}.
		\end{align*}
Thus we have a split short exact sequence of abelian groups 
		\[0\to G_{24}=\cok(\Sigma\phi_b^\sharp)\to [C_{\phi_b},S^6]\to \ker(\phi_b^\sharp)\to 0.\] 
 The proof of the Lemma is finished.
	\end{proof}
\end{lemma}

\begin{proposition}\label{prop:cohtp=5}
	Let $M$ be a closed orientable smooth $2$-connected $8$-manifold with $H_\ast(M)$ given by (\ref{table:HM}), $n=3$. If $T$ is $2$-torsion-free,
	then there is an isomorphism 
	\[\pi^5(M)\cong H^5(M)\oplus G_{24},\]
	where $G_{24}$ is a subgroup of $\Z/24$.

\begin{proof}
By Proposition \ref{prop:SM-8mfd-phi} and Lemma \ref{lem:Cphi-S}, we compute that 
	\begin{align*}
		[\Sigma M,S^6]\cong & \bigoplus_{i=1}^{l-c}[S^6,S^6]\oplus \bigoplus_{i=2}^{c}[C^6_\eta,S^6]\oplus [P^{6}(\frac{T}{\Z/3^{r_{j_1}}}),S^6]\oplus [C_{\phi},S^6]\\
		\cong &\bigoplus_{i=1}^{l-1}\Z\oplus \frac{T}{\Z/3^{r_{j_1}}}\oplus (\Z\oplus\Z/3^{r_{j_1}}\oplus G_{24})
		\cong \Z^{l}\oplus T\oplus G_{24}.
	\end{align*}	
	The proof is finished by applying the Freudenthal suspension isomorphism $[M,S^5]\cong [\Sigma M,S^{6}]$.

\end{proof}
\end{proposition}

We next provide a characterization of the cohomotopy group $\pi^3(M)$, analogous to the approach taken in Proposition \ref{prop:chtp=3} for $6$-manifolds.

\begin{proposition}\label{prop:cohtp=3}
	Let $M$ be a closed orientable smooth $2$-connected $8$-manifold with $H_\ast(M)$ given by (\ref{table:HM}), $n=3$. If $T$ is $2$-torsion-free, then there is an isomorphism
	\[\pi^3(M)\cong H^3(M)\oplus (\z{})^{l-c}\oplus G,\]
	where $G$ is an abelian group embedding in the short exact sequence
	\[0\to \z{}\to G\to \bigoplus_{i=2}^k\z{}\oplus \z{\delta}\to 0,\]
	where $\delta\in\{0,1\}$.

\begin{proof}
By Proposition \ref{prop:SM-8mfd-phi} and its proof given in Section \ref{sec:SM-proofs}, the suspension of the attaching map $\phi$ of the top cell of $M$ is given by the composition
		\[\phi_1\colon S^8\xra{\phi} V=S^5\vee S^4\vee C^6_\eta\vee P^6(3^{r_{j_0}})\vee P^6(3^{r_{j_1}})\hookrightarrow \Sigma \ol{M}.\]
	  Consider the induced exact sequences of groups 
	\[ [\Sigma V, BS^3]\xra{\Sigma \phi^\sharp}\pi_9( BS^3)\xra{q^\sharp}[C_{\phi}, BS^3]\xra{i^\sharp}[V, BS^3]\xra{\phi^\sharp} \pi_8( BS^3).\]	
	\[ [\Sigma^2\ol{M},BS^3]\xra{(\Sigma \phi_1)^\sharp}\pi_9(BS^3)\xra{(\Sigma q)^\sharp} [\Sigma M,BS^3]\xra{(\Sigma i)^\sharp}[\Sigma \ol{M},BS^3]\xra{\phi_1^\sharp}\pi_8(BS^3).\]
	Note that there are equalities 
	\[\cok\big((\Sigma\phi_1)^\sharp\big)=\cok\big((\Sigma\phi)^\sharp\big),\quad \ker(\phi_1^\sharp)\cong [\Sigma \ol{M}/V,BS^3]\oplus\ker(\phi^\sharp).\] 
	By Toda's results \cite{TodaBook}, we compute that  
	\begin{align*}
		[\Sigma V, BS^3]&\cong \pi_6( BS^3)\oplus \pi_5( BS^3)\oplus [C^7_\eta, BS^3]\oplus [P^7(3^{r_{j_1}}), BS^3]\\
		&\cong \z{}\lra{\iota \eta^2}\oplus \z{}\lra{\iota \eta}\oplus \Z/6\lra{\iota (\Sigma\nu')q_7^\eta}\oplus\Z/3\lra{\iota (\Sigma \alpha)q_7},\\
		[V, BS^3]&\cong \pi_5( BS^3)\oplus \pi_4( BS^3)\oplus [C^6_\eta, BS^3]\\
		&\cong \z{}\lra{\iota \eta}\oplus\Z\lra{\iota }\oplus\Z\lra{\iota \bar{\zeta}}.
	\end{align*}
	Since $\eta\nu_6=0$ (\cite[(5.9), page 44]{TodaBook}), $q_7^\eta i_5^\eta=0=q_7 i_6$, we see that 
	\[\Sigma\phi^\sharp=(x\cdot \nu_6+y\cdot \eta \nu_5+z\cdot i_5^\eta\nu_5\eta+w_{j_1}\cdot i_6\Sigma^3\alpha)^\sharp=0,\] 
	and hence $\cok((\Sigma\phi_1)^\sharp)=\cok(\Sigma\phi^\sharp)\cong\pi_9( BS^3)\cong \z{}$. 
	Using the formulae	$\phi^\sharp=(x\cdot \nu_5+y\cdot \eta \nu_5+z\cdot i_4^\eta\nu_4\eta)^\sharp$ and $\bar{\zeta}i_4^\eta\nu_4\eta=2\nu_4\eta=0$, we compute that 
	\[\ker(\phi^\sharp)=\Z\lra{\iota \bar{\zeta}}\oplus \left\{\begin{array}{ll}
		\Z\lra{\iota }& \text{ if $x$ is odd};\\
		\z{}\lra{\iota \eta}\oplus \Z\lra{2^y\iota }, y\in\{0,1\}&\text{ if $x$ is even}.
	\end{array}\right.\]
The suspension of the homotopy equivalence of $\ol{M}$ in Lemma \ref{lem:olM} with $n=3$ induces the following isomorphisms:
	\begin{align*}
		\ker(\phi_1^\sharp)\cong& [\Sigma \ol{M}/V,BS^3]\oplus\ker(\phi^\sharp)\\
		\cong& \bigoplus_{i=2}^{k} \pi_5(BS^3)\oplus \bigoplus_{i=2}^{l-c}\pi_4(BS^3)\oplus \bigoplus_{i=1}^{l-c}\pi_6(BS^3)\\
		&\oplus \bigoplus_{i=2}^c [C^6_\eta,BS^3] \oplus\ker(\phi^\sharp)\\
		\cong& \Z^l\oplus\bigoplus_{i=1}^{l-c}\z{}\lra{\iota\eta^2}\oplus\bigoplus_{i=2}^{k} \z{}\lra{\iota \eta}\oplus\z{\delta_x},
	\end{align*}
	where $\delta_x=1$ if $x$ is odd, otherwise $\delta_x=0$. 
	
	With $\ker(\phi_1^\sharp)$ computed above,  we get a short exact sequence 
	\[0\to \z{}\to [\Sigma M,BS^3]\to \ker(\phi_1^\sharp)\to 0.\]
 Since $\Aut(\z{})=\{id\}$, there is only one group action of $\ker(\phi_1^\sharp)$ on $\z{}$, i.e., the trivial action, and hence the group $[\Sigma M,BS^3]$ is abelian. The subgroup $\Z^l$ of $\ker(\phi_1^\sharp)$ clearly splits off $[\Sigma M,BS^3]$ and hence is a direct summand.
The pinch map $\tau_S\colon\Sigma M\to \bigvee_{i=1}^{l-c}S^6$ given by Proposition \ref{prop:SM-8mfd-phi} is a co-H-map and the induced homomorphism $\tau_S^\sharp$ is a splitting section of the homomorphism 
\[ [\Sigma M,BS^3]\to \ker(\phi_1^\sharp)\xra{\mathrm{proj.}}\bigoplus_{i=1}^{l-c}\z{}\lra{\iota\eta^2}.\]
Therefore, we get an isomorphism 
\[[\Sigma M,BS^3]\cong \Z^l\oplus \bigoplus_{i=1}^{l-c}\z{}\oplus G_x,\]
where $G_x$ is an abelian group characterized by the short exact sequence
\[0\to\z{}\to G_x\to \bigoplus_{i=2}^k\z{}\oplus \z{\delta_x}\to 0.\]
The proof of the Proposition is finished.
\end{proof}
\end{proposition}

Similar to Remark \ref{rmk:chtp=3}, we can not prove that the retraction map $\Sigma M\to \bigvee_{i=2}^k S^5$ given by Proposition \ref{prop:SM-8mfd-phi} is a co-H-map, so the group $\bigoplus_{i=2}^k\z{}$ may not splits off $G$.  The following example gives possible values of the direct summand $G=G_x$ of $\pi^3(M)$.

\begin{example}\label{lem:Cx-BS3}
If $M=S^4\cup e^8$, then there is a homotopy equivalence 
\[\Sigma M\simeq C_{x\cdot \nu_5}=S^5\cup_{x\cdot \nu_5}e^9\] 
 for some $x\in\Z/24$. There are isomorphisms
	\[\pi^3(M)\cong \left\{\begin{array}{ll}
	  \z{},& \text{ if } x\equiv 1\pmod 2;\\
	  	\z{}\oplus\z{},&\text{ if } x\equiv 0\pmod 4;\\
	  	\Z/4, & \text{ if } x\equiv 2\pmod 4.
	\end{array}\right.\] 
	
	\begin{proof}
		By Toda's computation \cite{TodaBook}, we have 
		\[\pi_9( BS^3)\cong \z{}\lra{\iota (\Sigma\nu')\eta^2},\quad \pi_8( BS^3)\cong \z{}\lra{\iota (\Sigma\nu')\eta=\iota \eta\nu_5}.\]
		The homotopy cofibre sequence $S^8\xra{x\cdot \nu_5} S^5\xra{i_5}C_{x\cdot \nu_5}\xra{q_9}S^9$ induces a short exact sequence of groups 
		\[0\to \z{}\to [C_{x\cdot \nu_5}, BS^3]\to \z{1-\delta_x}\to 0,\]
		where $\delta_x=1$ if $x$ is odd, otherwise $\delta_x=0$.
		
		If $x$ is odd, $[C_{x\cdot \nu_5}, BS^3]\cong\z{}$ and we are done. Now we assume that $x$ is even.  $\nu_5$ is a suspension implying that $x\cdot \nu_5=(x\cdot \id_5)\circ \nu_5$, and hence there is a homotopy commutative diagram that induces $\lambda(x)$:
		\[\begin{tikzcd}
			S^8 \ar[r,"\phi"]\ar[d,"\nu_5"]&S^5 \ar[d,equal]\ar[r,"i_5"]&C_{x\cdot \nu_5}\ar[d,"\lambda(x)"]\ar[r,"q_9"]&S^9\ar[d,"\nu_6"]\\
			S^5\ar[r,"x\cdot \id_5"]&S^5\ar[r,"i_5"]&P^6(x)\ar[r,"q_6"]&S^6
		\end{tikzcd}.\]
		Since $\pi_6( BS^3)\cong\pi_5( BS^3)\cong\z{}$, the above diagram induces a commutative diagram of short exact sequences of groups
		\[\begin{tikzcd}[column sep=small] 
			0\ar[r]& \pi_9( BS^3)\ar[r,"q_9^\sharp"]& {[C_{x\cdot \nu_5}, BS^3]}\ar[r,"i_5^\sharp"]& \z{}\ar[r]&0\\
			0\ar[r]& \pi_6( BS^3)\ar[u,"\nu_6^\sharp","\cong"swap]\ar[r,"q_6^\sharp"]& {[P^6(x), BS^3]}\ar[u,"\lambda(x)^\sharp"]\ar[r,"i_5^\sharp"]& \z{}\ar[u,equal]\ar[r]&0
		\end{tikzcd}.\]
		Thus for even $x$ we have $[C_{x\cdot \nu_5}, BS^3]\cong [P^6(x), BS^3]\cong [P^5(x), S^3]$, whose group structure follows by Lemma \ref{lem:htps:Moore} (\ref{Sn+2Pn+1}).
	\end{proof}	
\end{example}


\begin{proof}[Proof of Theorem \ref{thm:chtp-8mfd}]
The cohomotopy groups $\pi^i(M)$ for $i\geq 5$ are characterized by Corollary \ref{cor:chtp:2n+2} with $n=3$ and Proposition \ref{prop:cohtp=5}, the group  $\pi^3(M)$ refers to Proposition \ref{prop:cohtp=3}.
\end{proof}

\subsection{Cohomotopy of $3$-connected $10$-manifolds}\label{sec:chtp:n=4}
In this subsection we enumerate the cohomotopy set $\pi^i(M)$ of a $3$-connected $10$-manifold $M$ for $i=3,5$. 

\begin{lemma}\label{lem:10-mfd:MS3}
	Let $M$ be a $3$-connected 10-manifold with $H_\ast(M)$ given by (\ref{table:HM}), $n=4$. Let $c$ be given by Proposition \ref{prop:SM-10mfd}. 	There is an isomorphism 
	\[\pi^3(M)\cong \big(\bigoplus_{i=1}^{k+l-c}\z{}\big)\oplus\big(\bigoplus_{i=1}^{l-c}\Z/12\big)\oplus \big(\bigoplus_{i=1}^c\Z/6\big)\oplus (T\otimes \Z/3).\]
	\begin{proof}
	By Proposition \ref{prop:SM-10mfd}, the suspension of the attaching map of the top cell of $M$ is given by the composition
	\[\varphi_1\colon S^{10}\xra{\varphi} S^7\vee S^5\vee C^7_\eta\vee C^7_\eta\vee P^7(3^{r_{j_0}})\hookrightarrow \Sigma \ol{M},\]
	where $\varphi$	is given by (\ref{eq:varphi}). 
	Since $\pi_9(S^3)=\pi_{10}(S^3)=0$ \cite{TodaBook}, we have an isomorphism 
\[[\Sigma M,BS^3]\xra[\cong]{(\Sigma i)^\sharp} [\Sigma \ol{M},BS^3].\]	
It is easy to compute that $[C^6_\eta,S^3]\cong \Z/6$, $[P^6(p^r),S^3]\cong\Z/(3,p)$ for odd primes $p$.
The suspension of the homotopy equivalence of $\ol{M}$ in Lemma \ref{lem:olM} with $n=4$ induces the following isomorphisms:
\begin{align*}
	[\Sigma \ol{M},B S^3]&\cong \bigoplus_{i=1}^k \pi_6(B S^3)\oplus \bigoplus_{i=1}^{l-c}(\pi_5(B S^3)\oplus\pi_7(B S^3)) \\
	&\oplus \bigoplus_{i=1}^c [C^7_{\eta},B S^3] \oplus  [P^{6}(T),B S^3]\oplus [P^{7}(T),B S^3] \\
	&\cong \big(\bigoplus_{i=1}^{k+l-c}\z{}\big)\oplus\big(\bigoplus_{i=1}^{l-c}\Z/12\big)\oplus \big(\bigoplus_{i=1}^c\Z/6\big)\oplus (T\otimes \Z/3). 
\end{align*}
We then apply the natural isomorphism $[M,S^3]\cong [\Sigma M,B S^3]$ to complete the proof.
	\end{proof}
\end{lemma}

Next we enumerate $\pi^5(M)$ in terms of the EHP fibre sequence 
\[ \Omega^2S^6\xra{\Omega H}\Omega^2S^{11}\xra{P}S^5\xra{E}\Omega S^6\xra{H}\Omega S^{11},\]
where $H$ is the second James-Hopf map. There is an induced exact sequence of based sets 
\[[M,\Omega^2S^6]\xra{(\Omega H)_\sharp}[M,\Omega^2S^{11}]\xra{P_\sharp}\pi^5(M)\xra{E_\sharp}[M,\Omega S^6]\xra{H_\sharp}[M,\Omega S^{11}].\]
For an element $\bar{u}\in [M,\Omega S^6]$ that is a suspension, it is well-known that the canonical group action of $[M,\Omega^2S^{11}]$ on $E_\sharp^{-1}(\bar{u})$ is transitive; see also \cite[Theorem 3]{KMT2012}. Let $u\in\pi^{5}(M)$ be an element satisfying $E_\sharp(u)=\Sigma u=\bar{u}$. Then by \cite[Theorem 5.2]{Taylor2012}, there is a bijection between $E_\sharp^{-1}(\bar{u})$ and the cokernel of the homomorphism 
\begin{equation}\label{eq:varPhi}
	\varPhi=\varphi_u +(\Omega H)_\sharp\colon [M,\Omega^2S^6]\xra{}[M,\Omega^2S^{11}],
\end{equation}
where $\varphi_u(v)=v\wedge_w u$ is the composition 
\[M\xra{\Delta}M\wedge M\xra{v\wedge u}\Omega^2S^6\wedge S^5 \xra{\mathfrak{D}'}\Omega^2S^{11}\]
for some map $\mathfrak{D}'$ associated to the map $H$. We show the homomorphism $\varPhi$ is trivial by the following Lemmas \ref{lem:cv=0} and \ref{lem:MS5:JH}.

Since the double suspension $E^2\colon S^4\to \Omega^2S^6$ is $7$-connected and $\dim(M)=10$, there is a factorization of the map $\varphi_u(v)$: 
\begin{equation}\label{diag:MS5}
	\begin{tikzcd}
		M\wedge M\ar[r,"u\wedge v"]\ar[r]&S^5\wedge \Omega^2S^6\ar[r,"\mathfrak{D}'"]&\Omega^2S^{11}\\
		M\ar[u,"\Delta"]\ar[r,"\ol{\varphi}_u(v)"]&S^5\wedge S^4\ar[u,"\id_5\wedge E^2"]\ar[r,"c_v"]&S^9\ar[u,"E^2"]
	\end{tikzcd}, 
\end{equation}
where $c_v$ is the map of degree $c_v$. Note that the map $\ol{\varphi}_u$ is unique up to homotopy and the diagram (\ref{diag:MS5}) holds for any complexes $M$ of dimension at most $12$. 

Let $\caL X=\mathrm{map}(S^1,X)$ be the free loop space and let $\caL f\colon \caL X\to \caL Y$ be the induced map. Note that for an $H$-space $X$, the evaluation homotopy fibration 
$\Omega X\xra{j}\caL X\xra{\mathrm{ev}} X$ admits a homotopy section, and hence there is a homotopy equivalence $\caL X\simeq X\times \Omega X$.

\begin{lemma}\label{lem:cv=0}
The map $c_v$ in the diagram (\ref{diag:MS5}) has degree $0$, and hence $\varphi_u(v)=0$.
	\begin{proof}
	By the second homotopy commutative square of the diagram (\ref{diag:MS5}), we see that $c_v$ is determined by the induced homomorphism 
	\[\mathfrak{D}'_\ast\colon H_9(\Omega^2S^6\wedge S^5)\to H_9(\Omega^2S^{11}),\quad \mathfrak{D}'_\ast(a_4\otimes \iota_5)=c_v\cdot d_9,\]
	where $ a_4\in H_4(\Omega^2S^6),\iota_5\in H_5(S^5)$ and $d_9\in H_9(\Omega^2S^{11})$ are the chosen  generators. 	
	Let $b_5\in H_5(\Omega S^6)$ be the generator corresponding to $\iota_5\in H_5(S^5)$, then by \cite[Corollary 5.5]{Taylor2012}, the homomorphism (isomorphism)
	\[j_\ast\colon H_9(\Omega^2S^{11})\to H_9(\caL \Omega S^{11})\]
	sends $\mathfrak{D}'_\ast(a_4\otimes b_5)$ to $(\caL H)_\ast(a_4\otimes b_5)$.  

The homotopy commutative diagram 
\[\begin{tikzcd}
	\Omega^2 S^6\times \Omega S^6\ar[r,"\simeq"]\ar[d,"\Omega H\times H"]&\caL \Omega S^6\ar[d,"\caL H"]\\
	\Omega^2S^{11}\times \Omega S^{11}\ar[r,"\simeq"]&\caL \Omega S^{11}
\end{tikzcd}\]
induces the following commutative diagram of homology groups 
\[ \begin{tikzcd}
H_9(\Omega^2 S^6\times \Omega S^6)\ar[r,"\cong"]\ar[d,"(\Omega H\times H)_\ast"]&H_9(\caL \Omega S^6)\ar[d,"(\caL H)_\ast"]\\
	H_9(\Omega^2S^{11}\times \Omega S^{11})\ar[r,"\cong"]&H_9(\caL \Omega S^{11})
\end{tikzcd}.\]
Then the equality 
\[(\Omega H\times H)_\ast(a_4\otimes \iota_5)=(\Omega H)_\ast(a_4)\otimes H_\ast(\iota_5)=0\] 
implies that $(\caL H)_\ast(a_4\otimes b_5)=0$. Hence $\mathfrak{D}'_\ast(a_4\otimes b_5)=0$, that is, $c_v=0$.
	\end{proof}
\end{lemma}

\begin{lemma}\label{lem:MS5:JH}
The homomorphism	$(\Omega H)_\sharp\colon [C_{\varphi},\Omega S^6]\to [C_{\varphi},\Omega S^{11}]$ is trivial.
   \begin{proof}	
	   Let $\varphi\colon S^{10}\to W=S^7\vee S^5\vee C^7_\eta\vee C^7_\eta\vee P^7(3^{r_{j_0}})$ be given by (\ref{eq:varphi}).  By computation, we get a short exact sequence of abelian groups 
   \begin{equation}\label{ES:MS5}
	   0\to \Z/(x-z,2)\to [C_{\varphi},\Omega S^6]\xra{i_W^\sharp}[W,\Omega S^6]\to 0,
   \end{equation}
	   where $\Z/(x-z,2)$ is zero or generated by $\nu_6\nu_9$ by \cite{TodaBook},  
	   \[[W,\Omega S^6]\cong \pi_7(\Omega S^6)\oplus \pi_5(\Omega S^6)\oplus [C^7_\eta,\Omega S^6]\oplus [C^7_\eta,\Omega S^6]\cong\z{}\oplus\Z^3\]  consists of suspended elements. Then the naturality of $(\Omega H)_\sharp$ implies that it is zero after proving that the exact sequence (\ref{ES:MS5}) is split. 
   
   Note that the formulae $q_7^\eta \tilde{\nu}_7=-\nu_7$ and $\tilde{\zeta}\nu_7=-2\tilde{\eta}_7+\epsilon\cdot i_5^\eta\nu_5\eta^2$ in Lemma \ref{lem:htpgrps} (\ref{htpgrps:Ceta-unst3}) imply that $\Z/(x-z,2)=\z{}$  when $x$ and $z$ are both even, otherwise $\Z/(x-z,2)=0$. In the first case, it suffices to show that the short exact sequence 
   \[0\to \z{}\to [C_{x\cdot \nu_7},\Omega S^6]\xra{i_7^\sharp}\pi_7(\Omega S^6)\to 0\]
   is split when $x$ is even.
   By similar arguments to the proof of Lemma \ref{lem:Cx-BS3}, there is a map $\lambda(x)\colon C_{x\cdot \nu_7}\to P^8(x)$ such that $\lambda(x)i_7^\eta=i_7\in\pi_7(P^8(x))$ for even $x$.
   Then the formula in Lemma \ref{lem:htps:Moore} (\ref{Sn+2Pn+1}) implies that 
   \[\eta\bar{\eta}_{r_x}\lambda(x)i_7^\eta=\eta\bar{\eta}_{r_x}i_7=\eta^2\] 
   for some integer $r_x$. It follows that the above exact sequence for $[C_{x\cdot \nu_7},\Omega S^6]$ is split when $x$ is even. Therefore the short exact sequence (\ref{ES:MS5}) is always splitting and we complete the proof. 
   \end{proof}
   \end{lemma}

Summarizing the above discussion, we get the following. 
\begin{proposition}\label{prop:10mfd:MS5}
	Let $M$ be a $3$-connected 10-manifold with $H_\ast(M)$ given by (\ref{table:HM}), $n=4$. Let $u\in\pi^5(M)$ be a preimage of an element $\bar{u}\in \pi^6(\Sigma M)$ under the suspension map $E_\sharp\colon \pi^5(M)\to \pi^6(\Sigma M)$. Then there is a bijection between $E_\sharp^{-1}(u)$ and $\pi^9(M)\cong\z{}$. 
\end{proposition}

\begin{proof}[Proof of Theorem \ref{thm:chtp-10mfd}]
	Combine Corollary \ref{cor:chtp:2n+2} with $n=4$, Lemma \ref{lem:10-mfd:MS3} and Proposition \ref{prop:10mfd:MS5}.
\end{proof}

\section{Proofs of Propositions \ref{prop:SM-hbar}, \ref{prop:SM-8mfd-phi} and \ref{prop:SM-10mfd}}\label{sec:SM-proofs}

In this section we give the proofs of the suspension splitting of $(n-1)$-connected $(2n+2$)-manifolds given by Propositions \ref{prop:SM-hbar}, \ref{prop:SM-8mfd-phi} and \ref{prop:SM-10mfd}, respectively.

Throughout this section we assume that $T$ is $2$-torsion-free and set
\begin{equation}\label{eq:T}
	T=T_3\oplus T_{\geq 5},\quad T_3=\bigoplus_{j=1}^{t}\Z/3^{r_j}.
\end{equation} 
where $T_3$ is the $3$-primary component of $T$, $T_{\geq 5}=T/T_3$. Recall that for any integer $m\geq 3$, there is a homotopy equivalence (cf. \cite{Neisenbook})
\[P^{m}(T)\simeq P^{m}(T_3)\vee P^{m}(T_{\geq 5}) \simeq \bigvee_{j=1}^tP^m(3^{r_j}) \vee P^{m}(T_{\geq 5}).\]

Let $M$ be a closed orientable $(n-1)$-connected $(2n+2)$-manifold with $H_\ast(M)$ given by (\ref{table:HM}). 
Let $\ol{M}$ be the complement of the top cell of $M$. It is known that $\ol{M}$ is an $(n-1)$-connected $(2n+2)$-dimensional complex. Applying Chang's classification theorem (Lemma \ref{lem:Chang}), we have the following homotopy decomposition of $\ol{M}$.
\begin{lemma}\label{lem:olM}
	For $n\geq 3$, there is a homotopy equivalence 
	\[\ol{M}\simeq \bigvee_{i=1}^k S^{n+1}\vee \bigvee_{i=1}^{l-c}(S^n\vee S^{n+2})\vee \big(\bigvee_{i=1}^c C^{n+2}_\eta\big)\vee P^{n+1}(T)\vee P^{n+2}(T),\]
	where $c$ is an integer depending on $M$, $0\leq c\leq l$ and $c=0$ if and only if the Steenrod square $Sq^2$ acts trivially on $H^2(M;\z{})$.
	
	For $n=2$, the wedge sum in the above homotopy equivalence gives the homotopy type of the suspension $\Sigma \ol{M}$.

\end{lemma}

If $n=2$,  based on the Wall's connected sum decomposition \cite{Wall66}
\[M\approx M'\# \overset{k/2}{\underset{i=1}{\#}}(S^3\times S^3),\] 
where $M'$ is a simply connected 6-manifold with $H_3(M';\Q)=0$, 
Huang \cite[Corollary 2.2]{Huang-6mflds} proved that there is a homotopy equivalence 
\begin{equation}\label{eq:SM-SM'}
	\Sigma M\simeq \Sigma M'\vee \bigvee_{i=1}^{k/2}(S^4\vee S^4).
\end{equation}

\begin{remark}\label{rmk:pi-mfd}
	Fang and Pan \cite{FP04} extended Wall's splitting theorem to closed oriented smooth $(n-1)$-connected $(2n+2)$-dimensional $\pi$-manifolds (i.e., stably trivial manifolds) for $n\geq 2$ by proving that there is a homeomorphism $M\approx M'\# \overset{k/2}{\underset{i=1}{\#}}(S^{n+1}\times S^{n+1})$, where $M'$ is a $\pi$-manifold and $H_{n+1}(M';\Q)=0$. The statement when $H_\ast(M)$ is torsion-free is due to Ishimoto \cite{Ishimoto69}.
	 Using similar arguments to that of \cite[Corollary 2.2]{Huang-6mflds}, we see that there is  a homotopy equivalence 
	\[\Sigma M\simeq \Sigma M'\vee \bigvee_{i=1}^kS^{n+1},\] 
	where $M'$ is a $\pi$-manifold and $H_{n+1}(M';\Q)=0$.
\end{remark}

The remainder of this section is devoted to prove Propositions \ref{prop:SM-hbar}, \ref{prop:SM-8mfd-phi} and \ref{prop:SM-10mfd}.  We need the following matrix representation of maps between wedge sums of suspended spaces.
Let $X=\Sigma X'$ and $Y=\Sigma Y'$ be CW-complexes. A map  
$f+g\colon S^m\to X\vee Y$ that doesn't contain Whitehead product components can be denoted by the vector  $\smatwo{f}{g}$, where $f\colon S^6\to X$ and $g\colon S^m\to Y$. 
When $f$ and $g$ are suspended maps, or the attaching map of the homotopy cofibre $C_{f+g}$ doesn't contain Whitehead products, then elementary matrix row operations do not change the homotopy type of $C_{f+g}$; that is, in these two cases there are homotopy equalities 
	\[\mat{\id_X}{0}{h_{YX}}{\id_Y}\matwo{f}{g}=\matwo{f}{h_{YX}f+g},\quad  \mat{\id_X}{h_{XY}}{0}{\id_Y}\matwo{f}{g}=\matwo{f+h_{XY}g}{g},\]
where $h_{IJ}\colon J\to I$ are maps with $I,J\in\{X,Y\}$.

The following lemma is useful to study the homotopy type of a suspension space and will be frequently used in the proofs.

\begin{lemma}[Lemma 6.4 of \cite{HL-7mfd}]\label{lem:HL}
Consider a homotopy cofibre sequence 
\[S\xra{f}(\bigvee_{i=1}^nA_i)\vee B \xra{g}\Sigma C\] 
of simply connected spaces.
 Let 
 \[q_j\colon \big(\bigvee_{i}A_i\big)\vee B\to A_j,\text{ and }q_B\colon \big(\bigvee_{i}A_i\big)\vee B\to B\] 
be the canonical projections, where $j=1,\cdots,n$. Suppose that the composition $q_j\circ f$ is null homotopic for each $j\leq n$, then there is a homotopy equivalence
		\[\Sigma C\simeq \big(\bigvee_{i=1}^nA_i\big)\vee D,\]
		where $D$ is the homotopy cofibre of the composition $q_B\circ f$.

	\end{lemma}

\begin{proof}[Proof of Proposition \ref{prop:SM-hbar}]\label{proof:SM-hbar}
	Recall from (\ref{eq:SM-SM'}) that $\Sigma M\simeq \Sigma M'\vee \bigvee_{i=1}^kS^4$ for a $6$-manifold $M'$.  
By Lemma \ref{lem:olM}, the suspension of the $5$-skeleton $\ol{M'}$ is given by 
\begin{equation*}
	\Sigma\ol{M'}\simeq \bigvee_{i=1}^{l-c}(S^3\vee S^{5})\vee \big(\bigvee_{i=1}^c C^{5}_\eta\big)\vee P^{4}(T)\vee P^{5}(T).
\end{equation*}
From \cite[Section 4]{CS22}, we know that there is a homotopy equivalence 
\begin{equation}\label{eq:susp:SM'-Ch'}
	\Sigma M'\simeq P^4(T_{\geq 5})\vee P^5(T)\vee C_{h'},
\end{equation}
where $h'$ is the composition
\[h'\colon S^6\xra{h}\Sigma \ol{M'}\xra{\mathrm{pinch}} \bigvee_{i=1}^{l-c}S^3\vee \bigvee_{i=1}^{l-c}S^5\vee \bigvee_{j=1}^{t}P^4(3^{r_j})\vee \bigvee_{i=1}^cC^5_\eta=U'.\]
Moreover, the composition 
\[S^6\xra{h'}U'\twoheadrightarrow \bigvee_{i=1}^{l-c}S^3\vee \bigvee_{i=1}^{l-c}S^5\vee \bigvee_{j=1}^{t}P^4(3^{r_j})\]	
contains no Whitehead products. 

By Lemma \ref{lem:htpgrps}, there are isomorphisms 
		\begin{gather*}
			\pi_6(\Sigma S^2\wedge C^4_\eta)=\pi_6(\Sigma S^4\wedge C^4_\eta)=0,\\
			\pi_6(\Sigma P^3(p^r)\wedge C^4_\eta)\cong \pi_6(P^{6}(p^r)\vee P^{8}(p^r))\cong \pi_6(P^{6}(p^r))=0.
		\end{gather*}
Thus the attaching map $h'$ contains no Whitehead products.

By the groups in Lemma \ref{lem:htpgrps}, we can write 
	\begin{equation*}
		h'=\sum_{i=1}^{l-c}x_i\cdot \nu'+\sum_{i=1}^{l-c}y_i\cdot \eta+\sum_{j=1}^t z_j \cdot i_3^{r_j}\alpha+\sum_{j=1}^c w_j\cdot i_3^\eta\nu',
	\end{equation*}  
	where the coefficients 
	\begin{align*}
		&x_i\in\Z/12=\{0,\pm 1,\pm 2,\cdots,\pm 5,6\},~~y_i\in \z{}=\{0,1\},\\
		&z_j\in \Z/3=\{0,\pm 1\},~~w_j\in\Z/6=\{0,\pm 1,\pm 2,3\};
	\end{align*}
	the superscripts ``$r_j$'' and ``$\eta$'' are added to $i_3$ to distinguish the canonical inclusion maps $i_3\colon S^3\to X$ for $X=P^4(3^{r_j})$ and $C^5_\eta$, respectively. 
	
	Since $\eta=\eta_5$ and $i_3^{\eta}\nu'=\Sigma \pi_2$ are suspended generators, after applying matrix elementary row operations we can assume that 
	\begin{enumerate}
		\item $y_1\in\{0,1\}$ and $y_2=\cdots=y_{l-c}=0$; $y_1=0$ if and only if $M$ is spin.
		\item $w_1\in \{0,\pm 1,\pm 2,3\}$ and $w_2=\cdots=w_c=0$. 
	\end{enumerate}
Using the fact that the attaching map of $\Sigma M'$ contains no Whitehead products, we can assume that 
	\begin{enumerate}
		\setcounter{enumi}{2}
		\item $x_1\in\{0,\pm 1,\pm 2,\pm 3,\pm 4,\pm 5,6\}$ and $x_2=\cdots=x_{l-c}=0$. 
		\item $z_j=0$ for all $1\leq j\leq t$,  or $z_{j_0}=\pm 1$ and $z_j=0$ for $j\neq j_0$. This follows by the homotopy equality
		\[ \mat{\id_P}{0}{-B(\chi^r_s)}{\id_P}\matwo{i_3^r\alpha}{i_3^s\alpha}=\matwo{i_3^r\alpha}{0}\colon S^6\to P^4(3^r)\vee P^4(3^s) \text{ for }r\geq s.\]

	\end{enumerate}
	Then the statements in the Proposition follow by Lemma \ref{lem:HL}.
\end{proof}

\begin{proof}[Proof of Proposition \ref{prop:SM-8mfd-phi}]
By Lemma \ref{lem:olM}, the complement $\ol{M}$ of the top cell in $M$ is given by 
\[\ol{M}\simeq \bigvee_{i=1}^kS^4\vee \bigvee_{i=1}^{l-c}(S^3\vee S^{5})\vee \big(\bigvee_{i=1}^c C^{5}_\eta\big)\vee P^{4}(T)\vee P^{5}(T).\]
Let $\phi'\colon S^7\to \ol{M}$ be the attaching map of the top cell of $M$.
By Lemma \ref{lem:htpgrps} and \ref{lem:htps:Moore}, the following composition
		\begin{align*}
			&S^8\xra{\Sigma \phi}\Sigma \ol{M}\xra{\mathrm{pinch}} P^{a}(p^r) ~~(a=5,6) 
		\end{align*}
is null homotopic for primes $p\geq 5$, hence $P^5(T_{\geq 5})\vee P^6(T_{\geq 5})$ retracts off $\Sigma M$, by Lemma \ref{lem:HL}. Since $M$ is a smooth spin manifold, the secondary operation that detects $\eta^2$ acts trivially on $H^\ast(M;\z{})$ by \cite[Lemma 4.3]{LZ-5mfld}, it follows that the wedge summand $\bigvee_{i=1}^{l-c}S^6$ retracts off $\Sigma M$. Thus there is a homotopy decomposition 
\begin{equation*}
	\Sigma M\simeq \bigvee_{i=1}^{l-c}S^6\vee  P^5(T_{\geq 5})\vee  P^6(T_{\geq 5})\vee C_{\phi''},
\end{equation*}
	where the map $\phi''$ is the composition
	\[\phi''\colon S^8\xra{\Sigma \phi'}\Sigma\ol{M}\xra{\mathrm{pinch}} \bigvee_{i=1}^kS^5\vee \bigvee_{i=1}^{l-c}S^4\vee \bigvee_{i=1}^c C^{6}_\eta\vee \bigvee_{j=1}^t P^{5}(3^{r_j})\vee \bigvee_{j=1}^t P^{6}(3^{r_j})\]  given by the equation 
	\begin{multline}\label{eq:attachmap}
		\phi''=\sum_{i=1}^kx_i\cdot \nu_5 +\sum_{i=1}^{l-c}y_i\cdot \eta_4\nu_5 +\sum_{i=1}^{c}z_i\cdot i_4^\eta\nu_4\eta_7\\
		+\sum_{j=1}^t u_j\cdot \Sigma^2\tilde{\alpha}_{r_j}+\sum_{j=1}^tw_j\cdot i_5^{r_j}\Sigma^2\alpha,
	\end{multline}
	where $x_i\in\Z/24$, $y_i,z_i\in\z{}$, $u_j,w_j\in\Z/3$.

	Since all components of the map $\phi''$ are suspended maps, we can the method of matrix representation to reduce the expression (\ref{eq:attachmap}). By similar arguments to the proof of Proposition \ref{prop:SM-hbar}, we may assume that in the equation (\ref{eq:attachmap}) there hold:
	\begin{enumerate}
		\item $x_1\in\Z/24$, $x_2=\cdots=x_k=0$;
		\item $y_1\in\Z/2$, $y_2=\cdots=y_{l-c}=0$;
		\item $z_1\in\z{}$, $z_2=\cdots=z_c=0$;
		\item $u_{j_0}\in\Z/3=\{0,\pm 1\}$ and $u_j=0$ for $j\neq j_0$, where $r_{j_0}$ is the minimum of $r_j$ such that $u_j\neq 0$: this follows by the formula $B(\chi^r_s)\tilde{\alpha}_r=\tilde{\alpha}_s$ for $r\leq s$, see Lemma \ref{lem:htps:Moore} (\ref{htps:S7P4});
		\item $w_{j_1}\in\Z/3=\{0,\pm 1\}$ and $w_j=0$ for $j\neq j_1$, where $r_{j_1}$ is the maximum of $r_j$ such that $w_j\neq 0$. 
	\end{enumerate}
	The proposition then follows by Lemma \ref{lem:HL}.
\end{proof}

\begin{proof}[Proof of Proposition \ref{prop:SM-10mfd}]
	By Lemma \ref{lem:olM}, the complement $\ol{M}$ of the top cell in $M$ is given by 
\[\ol{M}\simeq \bigvee_{i=1}^kS^5\vee \bigvee_{i=1}^{l-c}(S^4\vee S^6)\vee \big(\bigvee_{i=1}^c C^6_\eta\big)\vee P^{5}(T)\vee P^{6}(T).\]
By Lemma \ref{lem:htpgrps} and \ref{lem:HL}, there is a homotopy decomposition 
\[\Sigma M\simeq \bigvee_{i=1}^kS^6\vee P^6(T)\vee P^7(T_{\geq 5})\vee C_{\varphi'},\]
where $\varphi'\colon S^{10}\to \bigvee_{i=1}^{l-c}(S^7\vee S^5)\vee \big(\bigvee_{i=1}^c C^7_\eta\big)\vee \bigvee_{j=1}^t P^7(3^{r_j})$ is given by 
\[\varphi=\sum_{i=1}^{l-c}x_i\cdot \nu_7+\sum_{i=1}^{l-c}y_i\cdot \nu_5\eta^2+\sum_{i=1}^{c}(z_i^1\cdot \Sigma\tilde{\nu}_6+z_i^2\cdot i_5^\eta\nu_5\eta^2)+\sum_{j=1}^t w_i\cdot \Sigma^4\alpha_{r_j},\]
where $x_i,z_i^1\in\Z/24$, $y_i,z_i^2\in\z{}$, $w_i\in\Z/3$.
Note that all components of $\varphi$ are suspensions, applying similar discussion on the coefficients as did in the case $n=3$, we may assume that
\begin{enumerate}[(i)]
	\item $x_1\in \Z/24$, $y_1\in\z{}$ and $x_i=y_j=0$ for $1<i,j\leq l-c$;
	\item $w_{j_0}\in\Z/3$, $w_j=0$ for $j\neq j_0$;
	\item $z_1^1\in\Z/24$ and $z_1^2=1$, $z_i^1=z_i^2=0$ for $i=2,\cdots,c$; or $z_1^1\in\Z/24$, $z_2^2=1$, $z_i^1=0$ for $2\leq i\leq c$,  and $z_i^2=0$ for $2\neq i\leq c$. Note that a necessary condition of the second case is $c\geq 2$. Moreover, 
\[(y,z_1^2,z_2^2)\in \{(0,0,0),(1,0,0),(0,1,0),(0,0,1)\}.\]
\end{enumerate}
Then the proposition follows by Lemma \ref{lem:HL}. 
\end{proof}

\bibliography{cohtp-2n+2}
\bibliographystyle{amsplain}

\end{document}